\documentclass[11pt]{amsart}
\usepackage{mathptmx}
\usepackage{amsmath,amssymb,tikz-cd}
\usepackage{quiver}
\usepackage{amsmath}
\usepackage{amssymb}
\usepackage{amsthm}
\usepackage{xspace}
\usepackage[all,tips]{xy}
\usepackage{syntonly}
\usepackage{hyperref}
\usepackage{arydshln}
\usepackage{todo}
\usepackage{enumerate}
\usepackage{tikz-cd}
\usepackage{tikz}
\usepackage{adjustbox}
\usepackage{comment}

\newtheorem{thm}{Theorem}[section]
\newtheorem{cor}[thm]{Corollary}
\newtheorem{prop}[thm]{Proposition}
\newtheorem{lemma}[thm]{Lemma}

\theoremstyle{definition}
\newtheorem{defn}[thm]{Definition}

\theoremstyle{remark}
\newtheorem{rem}{Remark}

\DeclareMathOperator{\Aut}{Aut} \DeclareMathOperator{\Out}{Out}
 
\DeclareMathOperator{\GL}{GL}

\DeclareMathOperator{\Inn}{Inn}
\DeclareMathOperator{\Fitt}{Fitt}\DeclareMathOperator{\Hom}{Hom}

\newcommand{\N}{\ensuremath{\mathbf{N}}}
\newcommand{\U}{\ensuremath{\mathbf{U}}}

\newcommand{\R}{\ensuremath{\mathbb{R}}}
\newcommand{\Q}{\ensuremath{\mathbb{Q}}}
\newcommand{\Z}{\ensuremath{\mathbb{Z}}}
\newcommand{\C}{\ensuremath{\mathbb{C}}}
\newcommand{\F}{\ensuremath{\mathbb{F}}}
\newcommand{\G}{\ensuremath{\mathbf{G}}}

\newcommand{\T}{\ensuremath{\mathbf{T}}}
\newcommand{\HH}{\ensuremath{\mathbf{H}}}

\title[Rational cohomology of Zariski dense subgroups]{Rational cohomology and Zariski dense subgroups of Solvable Linear Algebraic Groups}
\author{Milana Golich}
\address{Department of Mathematics\\Purdue University\\West Lafayette, IN 47907}
\email{mgolich@purdue.edu}

\author{Antonio L\'{o}pez Neumann}
\address{Université Paris Cité, Sorbonne Université, CNRS, IMJ-PRG, F-75013 Paris, France}
\email{lopezneumann@imj-prg.fr}

\author{Mark Pengitore}
\address{Department of Mathematics\\University of Virginia\\Charlottesville, VA 22903}
\email{waj9cr@virginia.edu}

\begin{document}

\begin{abstract}
We show that for an irreducible solvable $\Q$-defined linear algebraic group $\G$, there exists an isomorphism between the cohomology rings with coefficients in a finite dimensional rational $\G$-module $M$ of the associated $\Q$-defined Lie algebra $\mathfrak{g_\Q}$ and Zariski dense subgroups $\Gamma \leq \G(\Q)$ that satisfy the condition that they intersect the $\Q$-split maximal torus discretely. We further prove that the restriction map in rational cohomology from $\G$ to a Zariski dense subgroup $\Gamma \leq \G(\Q)$ with coefficients in $M$ is an injection. As an application, we show that if the $\Q$-algebraic hull of $\Gamma$ is irreducible and its Fitting subgroup is $S$-arithmetic, then the action of $\Out(\Gamma)$ on its cohomology with $\Q$-coefficients is $S$-arithmetic.
\end{abstract}

\maketitle


\section{Introduction}\label{Section: Intro}

Let $\G$ be a linear algebraic group defined over a field $k$.  A linear representation $\rho$ of $\G$ on a finite dimensional $k$-vector space $V$ is $\emph{rational}$ if $\rho: \G \to \GL(V)$ is a $k$-rational homomorphism. Equivalently, we refer to $V$ as a finite dimensional $\emph{rational}$ $\G$-\emph{module}. The \emph{rational cohomology groups} of $\G$ with coefficients in $V$ are denoted by $H^n(\G, V)$ (see Section \ref{Section: Cohom} for a definition).  By \cite{borel}, $\G$ has an associated $k$-defined Lie algebra which we denote by $\mathfrak{g}_k$. Since the $k$-rational homomorphism $\rho \colon \G \to \GL(V)$ naturally induces a Lie algebra homomorphism $d\rho: \mathfrak{g}_k \to \mathfrak{gl}(V)$ where $d$ is the differential, we have that $V$ is a $\mathfrak{g}_k$-module derived from a rational representation. The Lie algebra cohomology groups of $\mathfrak{g}_k$ with coefficients in $V$ are then denoted $H^n(\mathfrak{g}_k, V)$ (see Section \ref{Section: Cohom} for a definition). Lastly, given a subgroup $\Gamma \leq \G(k)$ and a rational $\G$-module $V$, we can see $V$ as a $\Gamma$-module just by restricting the action to $\G$. We can thus study the usual \emph{group cohomology} spaces of $\G$ with coefficients in $V$, denoted by $H^*(\Gamma, V)$. The goal of this article is to understand the relationship between these three different cohomology theories for the case of solvable linear algebraic groups. 

In \cite{Mostow}, Mostow proved that there exists an isomorphism between $H^*(\mathfrak{g}_\R, \R)$ and $H^*(\Gamma, \R)$ when  $G$ is a connected, simply connected real solvable Lie group and $\Gamma \leq G$ is a discrete cocompact subgroup such that $\textrm{Ad}_{{\mathfrak{g}_\R}}(G)$ and $\textrm{Ad}_{{\mathfrak{g}_\R}}(\Gamma)$ have the same algebraic hulls. Kunkel \cite{Kunkel} then provided an analogue of Mostow's result in the case of when $\G$ is an irreducible solvable $\Q$-defined linear algebraic group. Specifically, he showed that if $\Gamma$ is a discrete, cocompact arithmetic subgroup of the identity component $\G(\R)^0$ of the real points of $\G$, then the pair ($\G(\R)^0, \Gamma$) satisfies the conditions of Mostow's theorem, and hence, $H^*(\mathfrak{g}_\R, \R) \cong H^*(\Gamma, \R)$. Then, by a spectral sequence argument, Kunkel proved that $H^*(\mathfrak{g}_\Q, \Q) \cong H^*(\Gamma, \Q)$ as cohomology rings. 

Our main theorem furthers these results for the case of when $\G$ is an irreducible solvable $\Q$-defined linear algebraic group and $\Gamma \leq \G(\Q)$ is a Zariski dense subgroup such that its intersection with the $\Q$-split maximal torus is discrete in the manifold topology on $\G(\R)$. Given that $\G(\R)$ is a connected solvable Lie group, we refer to the Lie group topology on $\G(\R)$ as the Euclidean topology.

\begin{thm}\label{main_thm}
         Let $\G$ be an irreducible solvable $\Q$-defined linear algebraic group  with associated $\Q$-defined Lie algebra $\mathfrak{g}_{\Q}$, and let $\Gamma \leq \G(\Q)$ be a Zariski dense subgroup such that it intersects the $\Q$-split maximal torus discretely in the Euclidean topology. Then for every finite dimensional rational $\G$-module $M$, there exists an isomorphism $\Phi_M: H^*(\mathfrak{g}_{\Q}, M) \to H^*(\Gamma, M)$ of cohomology rings.
\end{thm}

    Our result differs from Kunkel's in that we do not assume that our Zariski dense subgroup $\Gamma$ is necessarily discrete in $\G(\R)$ with the Euclidean topology nor finitely generated. In particular, this provides a vast generalization of Kunkel's work to the nondiscrete setting. Examples to keep in mind are finitely generated torsion-free solvable groups of finite abelian ranks that are not virtually polycyclic such as the solvable Baumslag-Solitar groups and more generally finitely generated solvable groups that are $\mathbb{Q}$-linear but not linear over $\mathbb{Z}$. Such examples do not fall under the hypotheses of Kunkel's work. 
    
Additionally, our result generalizes Kunkel's result from cohomology with trivial coefficients to cohomology with coefficients in any finite dimensional rational $\G$-module.


We observe in the following series of remarks that we cannot strengthen Theorem \ref{main_thm} or conversely, drop any assumption. 

\begin{rem}
     We note that we require the irreducible  assumption for Theorem \ref{main_thm} to hold. For example, $\Z \rtimes \Z / 2 \Z$ is Zariski dense in $\Q \rtimes \Z / 2 \Z$, but these groups do not have isomorphic cohomology rings over $\Q$.
\end{rem}

\begin{rem}
We also note that we cannot extend Theorem \ref{main_thm} to Zariski dense subgroups of $k$-defined irreducible solvable linear algebraic groups where $k$ is any characteristic 0 field. For example, the finitely generated solvable group $\Gamma = \Z \wr \Z$ is not $\Q$-linear but is $\R$-linear with representation given by
$$
\Z \wr \Z
=
\left\{ \begin{bmatrix}
    \pi^n & f \\
    0 & 1
\end{bmatrix}
:
f \in \Z[\pi, \pi^{-1}], n \in \Z
\right\}
$$
where $\Z[\pi, \pi^{-1}]$ is the ring of Laurent polynomials with integer coefficients in one variable evaluated at $\pi$. It follows that $\Gamma$ is Zariski dense in the $\R$-defined linear algebraic group $\R \rtimes \R_{>0}$ where $\R_{>0}$ acts on $\R$ by multiplication. Since $\bigoplus_{i=1}^\infty \Z$ is a subgroup of $\Gamma$, it follows that $\Gamma$ has infinite cohomological dimension over $\R$. However, the Lie algebra of $\R \rtimes \R$ has finite cohomological dimension over $\R$ since it is a subalgebra of $\mathfrak{gl}_2(\R)$. Hence, for an irreducible solvable linear algebraic group to impose cohomological finiteness conditions on $\Gamma$, it is necessary that $\Gamma$ is $\Q$-linear.
\end{rem}

\begin{rem}
We further have that Theorem $\ref{main_thm}$ does not hold for arbitrary Zariski dense subgroups of irreducible solvable $\Q$-defined linear algebraic groups. In fact, $S$-arithmetic lattices provide counterexamples in the following way. Given an irreducible solvable $\Q$-defined linear algebraic group $\G = \U \rtimes \T$ where $\U$ is the $\Q$-defined unipotent radical and $\T$ is the $\Q$-split maximal torus, we have that $\G(\Z[\frac{1}{S}])$ is Zariski dense for all sets of primes but $\G(\Z[\frac{1}{S}]) \cap \T(\Q)$ is not discrete in $\T(\R)$ in the Euclidean topology. In this case, $\G(\mathbb{Z}[\frac{1}{S}])$ has cohomological dimension equal to its Hirsch length. While the Hirsch length of $\U(\Z[\frac{1}{S}])$ is bounded by $\dim_\Q(\U)$, the Hirsch length of $\T(\Z[\frac{1}{S}])$ is given by $\dim_{\Q}(\T) \cdot |S|$.  Therefore, for larger and larger sets of primes $S$, $\G(\Z [\frac{1}{S}])$ has larger and larger Hirsch length. This implies that the groups $\G(\Z[\frac{1}{S}])$ for varying sets of primes cannot have isomorphic cohomology rings over $\Q$.
\end{rem}

\begin{rem}
    We conclude by observing that Theorem $\ref{main_thm}$ does not hold if we drop the linear algebraic group structure on $\G$ and simply consider $\G$ as a connected solvable Lie group and $\Gamma$ as a cocompact lattice. For example, if we realize $\C$ as $\R + i \R$ and assume that the one parameter group $X(t)$ acts by multiplication by $e^{it}$, we have that $\G = \R^2 \rtimes \R$ is a solvable Lie group which does not admit a linear algebraic group structure. We also have that $(\Z + i \Z) \rtimes \Z$ is a cocompact lattice in $\G$ that is isomorphic to $\Z^3$ which is not Zariski dense. By direct computation, one can see that the lattice $\Z^3$ and the Lie algebra $\mathfrak{g}$ do not have isomorphic cohomology rings over $\Q$. Hence, we require the linear algebraic group structure on $\G$.
\end{rem}

\begin{rem}
    We also refer the reader to Kasuya \cite[Theorem 1.3]{kasuya1} who derived these results when $\Gamma$ is a full subgroup of $\textbf{G}$.
\end{rem}

Additionally, we prove the following theorem regarding the restriction map in rational cohomology from an irreducible solvable $\Q$-defined linear algebraic group to an arbitrary Zariski dense subgroup.  

\begin{thm}\label{resmap_solv}
Let $\G$ be an irreducible solvable $\Q$-defined linear algebraic group and $\Gamma \leq \G(\Q)$ be a Zariski dense subgroup. Then for every finite dimensional rational $\G$-module $M$, the restriction map $r_{M} \colon H^*(\G, M) \to H^*(\Gamma, M)$ is an injection of cohomology rings.
\end{thm}

\subsection{Applications} 
We derive several applications of Theorem \ref{main_thm}. We first relate the cohomology rings of suitable Zariski dense subgroups of irreducible solvable $\Q$-defined linear algebraic groups to the cohomology rings of their finite index subgroups. We then determine the arithmetic properties of the representation of $\Out(\Gamma)$ on $H^*(\Gamma,\C)$ when $\Gamma$ is a torsion-free, finitely generated solvable group of finite abelian ranks. Finally, we show that there exist countably many examples of pairwise non-isomorphic, non-commensurable $\Q$-linear solvable groups with isomorphic cohomology rings over $\Q$. In particular, solvable Baumslag-Solitar groups yield such examples. 

\subsubsection{Finite index subgroups of Zariski dense subgroups of irreducible solvable $\Q$-defined linear algebraic groups}

By \cite{borel}, it follows that every closed subgroup of finite index of a linear algebraic group contains the connected component of the identity.  Therefore, given the irreducible assumption of Theorem $\ref{main_thm}$, we have the following immediate corollary. 

\begin{cor}\label{cor:finite_index}
    Let $\G$ be an irreducible solvable $\Q$-defined linear algebraic group  with associated $\Q$-defined Lie algebra $\mathfrak{g}_{\Q}$, and let $\Gamma \leq \G(\Q)$ be a Zariski dense subgroup such that it intersects the $\Q$-split maximal torus discretely in the Euclidean topology. Then for every finite dimensional rational $\G$-module $M$ and finite index subgroup $\Delta \leq \Gamma$, we have $H^*(\Delta, M)$ is isomorphic to $H^*(\Gamma, M)$ as cohomology rings.
\end{cor}

\begin{rem}
We note that this does not hold true in the setting of Zariski dense subgroups of connected semisimple Lie groups defined over $\Q$. By \cite[Theorem 1.1]{strongly_dense_free_subgroups}, it follows that every connected semisimple $\Q$-defined linear  algebraic group contains a Zariski dense free subgroup $F_2$. Since the cohomology rings of non-abelian free groups of different ranks are pairwise non-isomorphic with trivial $\Q$-coefficients, Corollary \ref{cor:finite_index} fails in this setting. 
\end{rem}

\subsubsection{Arithmetic Constructions} We refer the reader to \cite{dere_pengitore} for more details regarding the following discussion.

Let $\Gamma$ be a torsion-free, finitely generated solvable group of finite abelian ranks. By \cite[Theorem 1.3]{dere_pengitore}, there exists a faithful representation $\varphi \colon \Gamma \to \GL(n,\Q)$ satisfying the following properties. First, it follows that the Zariski closure of $\varphi(\Gamma)$, denoted by $\HH_\Gamma$, is such that the centralizer of the unipotent radical $\U(\HH_\Gamma)$ is contained in $\U(\HH_\Gamma)$ and $\dim_\Q(\U(\HH_\Gamma)) = h(\Gamma)$ where $h(\Gamma)$ is the Hirsch length of $\Gamma$. Moreover, there exists a finite set of primes $S$ such that $\Gamma \cap \HH_\Gamma(\Z[\frac{1}{S}])$ has finite index in $\varphi(\Gamma)$. Specifically, $S$ is the set of primes $p$ for which there exists a pair of subgroups $N \trianglelefteq M$ such that $M/N$ is a $p$-quasicyclic group (see \cite{lennox_robinson} for more details). We then say a group $\Gamma$ satisfying the above conditions has a \emph{$\Q$-algebraic hull} given by $\HH_\Gamma$. Finally \cite[Theorem 1.3]{dere_pengitore} implies that the pair $(\HH_\Gamma, \varphi)$ is well defined up to $\Q$-isomorphism and that each automorphism of $\Gamma$ extends to a unique $\Q$-defined automorphism of $\HH_\Gamma$. We then note that $\Aut(\Gamma)$ is a subgroup of the group of algebraic automorphisms $\mathcal{A}_\Gamma = \Aut_a(\HH_\Gamma)$ of $\HH_\Gamma$.

The automorphism group $\Aut(\Gamma)$ admits an action on $H^*(\Gamma, \Q)$, and since the inner automorphisms act trivially on $H^*(\Gamma, \Q)$, the outer automorphism group $\Out(\Gamma)$ is naturally represented on the cohomology ring $H^*(\Gamma, \Q)$.  When the $\Q$-algebraic hull $\HH_\Gamma$ is irreducible, Theorem \ref{main_thm} implies that $H^*(\mathfrak{h},\Q)$ is isomorphic to $H^*(\Gamma, \Q)$ as cohomology rings where $\mathfrak{h}$ is the Lie algebra of $\HH_\Gamma$. We note that $\Inn_{\HH_\Gamma}$ is a $\Q$-closed subgroup, and thus, we then call the quotient
$$
\Out_a(\HH_\Gamma) = \mathcal{A}_\Gamma / \Inn_{\HH_\Gamma}
$$
the group of algebraic outer automorphisms of $\HH_R$. It is again a $\Q$-defined linear algebraic where $\Aut(\Gamma) \leq \mathcal{A}_\Gamma$, and we obtain a group homomorphism
$$
\pi_\Gamma \colon \Out(\Gamma) \to \Out_a(\HH_\Gamma)
$$
by restricting the quotient homomorphism $\mathcal{A}_\Gamma \to \Out_a(\HH_\Gamma)$ to the subgroup $\Aut(\Gamma) \leq \mathcal{A}_\Gamma$. We may view $\GL(H^*(\Gamma, \Q))$ as a $\Q$-defined linear algebraic group since $\mathfrak{h}$ is $\Q$-defined, and therefore, there exists an induced $\Q$-defined homomorphism $\rho \colon \Out(\HH_\Gamma) \to \GL(H^*(\Gamma, \Q))$. We see that the finite dimensional rational vector space $H^*(\Gamma, \Q)$ admits a $\Z[\frac{1}{S}]$-structure given by the image of the base change homomorphism $H^*(\Gamma, \Z[\frac{1}{S}]) \to H^*(\Gamma, \Q)$. We then see that representation
\begin{center}
$\rho \circ \pi_\Gamma \colon \Out(\Gamma) \to \GL(H^*(\Gamma, \Q))$
\end{center}
is then integral with respect to the above $\Z[\frac{1}{S}]$-structure. By \cite[Theorem C]{dere_pengitore}, we then see that the image of $\Out(\Gamma)$ in $\Out(\HH_{\Gamma})$ is a $S$-arithmetic lattice in its Zariski closure. Since the image of a $S$-arithmetic lattice under a $\Q$-defined homomorphism is a $S$-arithmetic lattice, we have the following theorem.
\begin{thm}\label{MT3}
     Let $\Gamma$ be a  torsion-free, finitely generated solvable group of finite abelian ranks with spectrum $S$ such that its $\Q$-algebraic hull is irreducible and that $\Fitt(\Gamma)$ is $S$-arithmetic. Then the image of $\Out(\Gamma)$ under $\rho \circ \pi_\Gamma$  is a $S$-arithmetic lattice in its Zariski closure in $\GL(H^*(\Gamma, \Q)).$ 
\end{thm}

\begin{rem}
In the case of when $\Gamma$ is polycyclic, Baues and Grunewald \cite[Theorem 1.13]{baues_grunewald} employ geometric methods in \cite{bause} to construct a homomorphism $\eta \colon \Out(\textbf{H}_\Gamma) \to \GL(H^*(\Gamma, \C))$ which extends the natural representation $\rho \colon \Out(\Gamma) \to \GL(H^*(\Gamma, \C))$. In particular, they demonstrate that the image of the representation $\rho$ is arithmetic in its Zariski closure in $\GL(H^*(\Gamma, \C)).$  If $\Gamma$ is not necessarily polycyclic,  the methods of Baues do not generalize to the more general collection of torsion-free, finitely generated solvable groups of finite abelian ranks with irreducible $\Q$-algebraic hull. While $\Gamma$ admits an affine action on $\textbf{U}(\textbf{H}_\Gamma)$, the action of $\Gamma$ is not properly discontinuous. It then follows that $\U(\textbf{H}_\Gamma)/\Gamma$ is not a $K(\Gamma,1).$ While automorphisms of $\Gamma$ will descend to $\U(\textbf{H}_\Gamma)/\Gamma$, they will not induce an isomorphism of $H^*(\Gamma, \Q)$, and this action need not agree with the representation $\rho$.
\end{rem}

Given that the stabilizers of cohomology classes in $H^*(\Gamma, \Q)$ are $S$-arithmetic, we have the following immediate corollary.

\begin{cor}
    Let $\Gamma$ be a  torsion-free, finitely generated solvable group of finite abelian ranks such that its $\Q$-algebraic hull is irreducible and that $\Fitt(\Gamma)$ is $S$-arithmetic. If $[\alpha] \in H^*(\Gamma, \Q)$ is a fixed cohomology class, then $\text{Stab}_{\rho \circ \pi_\Gamma(\Out(\Gamma))}([\alpha])$ is $S$-arithmetic in its Zariski closure in $\GL(H^*(\Gamma, \Q))$.
\end{cor}

\subsubsection{$S$-arithmetic lattices in a fixed $\Q$-defined solvable linear algebraic group} 
    
Given a torsion-free, finitely generated nilpotent group $\Delta$, we denote the $\Q$-algebraic hull by $\HH_\Delta$, and we note that it is irreducible. We then have by \cite[Proposition 3.5]{dere_pengitore} that the image of any injective endomorphism $\varphi \colon \Delta \to \Delta$ has finite index in $\Delta.$ \cite{dere_pengitore} also shows that the group of automorphisms $\Aut(\HH_\Delta)$ is a $\Q$-defined linear algebraic group. Letting $\Psi_\varphi$ be the unique extension of $\varphi$ to an automorphism of $\HH_\Delta$, we let $\overline{\left<\Psi_{\varphi}^\ell \: : \: \ell \in \Z  \right>}$ be the Zariski closure of $\left<\Psi_{\varphi}^\ell \: : \: \ell \in \Z  \right>$ in $\Aut(\HH_\Delta)$. When $\overline{\left<\Psi_{\varphi}^\ell \: : \: \ell \in \Z  \right>}$ is irreducible, it follows that there exists an injective $\mathbb{Q}$-defined homomorphism $\lambda \colon \mathbb{Q} \to \Aut(\HH_\Delta)$ whose image is $\overline{\left<\Psi_{\varphi}^\ell \: : \: \ell \in \Z  \right>}$ where $\lambda(m) = \Psi_{\varphi}^m$ for $m \in \Z .$ We then have that any group of the form
    \begin{center}
    $\Delta_\varphi = \left<\Delta,t \: | \: t mt^{-1} = \varphi(m), m \in \Delta \right>$ 
    \end{center}

    is a finitely generated Zariski dense subgroup of $\HH_\Delta \rtimes_\lambda \mathbb{Q}$. We also have that the group given by 
    \begin{center}
    $\Delta_{\varphi^s} = \left<\Delta, t \: | \: t mt^{-1} = \varphi^s(m), m \in \Delta \right>$
    \end{center}
    for $s \in \Z$ is Zariski dense in $\HH_\Delta \rtimes_\lambda \mathbb{Q}.$ We thus have the following theorem.
    
   \begin{thm}\label{S arithmetic thm}
       Let $\Delta$ be a torsion-free, finitely generated nilpotent group with an injective endomorphism $\varphi \colon \Delta \to \Delta$, and let $\Psi_\varphi$ be the unique extension of $\varphi$ to an automorphism of the $\Q$-algebraic hull $\HH_\Delta$ such that the Zariski closure of $\left<\Psi_\varphi^\ell \: : \: \ell \in \mathbb{N} \right>$ is irreducible in $\Aut(\HH_\Delta)$. Then, $H^*(\Delta_{\varphi^s}, \Q)$ is isomorphic to $H^*(\Delta_{\varphi^\ell}, \Q)$ as cohomology rings for all $s, \ell \in \mathbb{N}$.
   \end{thm}

     In general, we do not know the isomorphism type of $\Delta_{\varphi^s}$ for a given $s$. However, pairwise non-isomorphic examples provided by the above construction include the Baumslag-Solitar groups $BS(1,n)$ which are given by the presentation
    \begin{center}
        $BS(1,n) = \left<x,y \: | \: x^{-1}yx = y^n \right>$.
    \end{center}

In particular, $BS(1,n)$ is Zariski dense in $\Q \rtimes \Q$ for all natural numbers $n \geq 2$. By Theorem \ref{main_thm}, we see that $H^*(BS(1,s),\Q)$ is isomorphic to $H^*(BS(1,\ell),\Q)$ as cohomology rings for all natural numbers $s, \ell \geq 2$. By \cite[Theorem 1.2]{commensurable_BS_groups}, the Baumslag-Solitar groups $BS(1,n)$ are non-commensurable when $s$ and $\ell$ are not common powers of the same integer. Therefore, we construct countably many examples of pairwise non-isomorphic, non-commensurable Baumslag-Solitar groups with isomorphic cohomology rings over $\Q$. While this fact is well known using topological techniques, we provide an algebraic group theoretic proof of this classic computation.

    We now turn to an additional application of Theorem \ref{S arithmetic thm}. Let $R$ be a subring of $\Q$, and let $H(R)$ be the Heisenberg group which is given by
    $$
    H(R) = 
    \left\{ \begin{bmatrix}
        1 & a & b \\
        0 & 1 & c\\
        0 & 0 & 1
    \end{bmatrix} : a,b,c \in R\right\}.
    $$ For a rational number $r > 0$, we define the map $\delta_r \colon H(\Q) \to H(\Q)$ given by
    $$
\delta_r\left( \begin{bmatrix}
        1 & a & b \\
        0 & 1 & c\\
        0 & 0 & 1
    \end{bmatrix}\right) = \begin{bmatrix}
        1 & ra & r^2b \\
        0 & 1 & rc\\
        0 & 0 & 1
    \end{bmatrix}.
    $$
    It follows that $\Z$ acts on $H(\Q)$ via $\ell \to \delta_r^\ell$ for $\ell \in \Z.$ The group $\overline{\left< \delta_r^\ell \: : \: \ell \in \Z\right>}$ in $\Aut(H(\Q))$ is given by $\Q$, and thus, $H(\Q) \rtimes \Q$ is an irreducible solvable $\Q$-defined linear algebraic group.
    
    For a natural number $n > 1$, the group $H(\Z)_{\delta_n}$ is an ascending HNN extension that is Zariski dense in $H(\Q) \rtimes \Q$. Letting $\mathfrak{h}$ be the Lie algebra of $H(\Q)$, we have that 
    $$
d\delta_n\left( \begin{bmatrix}
        0 & a & b \\
        0 & 0 & c\\
        0 & 0 & 0
    \end{bmatrix}\right) = \begin{bmatrix}
        0 & na & n^2b \\
        0 & 0 & nc\\
        0 & 0 & 0
    \end{bmatrix}
    $$
    where $d$ is the differential. Thus, for $n > 1$, we have that $\mathfrak{h} \rtimes_{d \delta_n} \Q$ are isomorphic as $\Q$-defined Lie algebras. Since $H(\Z)_{\delta_n}$ is Zariski dense in $H(\Q) \rtimes \Q$ for all $n$, Theorem \ref{S arithmetic thm} implies that they have isomorphic cohomology rings over $\Q$.

We conclude this section with a few additional remarks regarding Theorem \ref{main_thm}.

\begin{rem}
    We first observe that the irreducible assumption of $\overline{\left<\Psi_{\varphi}^\ell \: : \: \ell \in \Z  \right>}$ is necessary. For instance, let $\varphi \colon \Z^2 \to \Z^2$ be the automorphism given by the switching of coordinates. Since $\varphi$ has order two, the image of $\varphi$ in $\GL_2(\Q)$ is not irreducible. We then have that $\Q^2 \rtimes_\varphi C_2$ where $C_2$ is the cyclic group of order two acting on $\Q^2$ by $\varphi$ does not contain $\Z^2 \rtimes_\varphi \Z$ as a Zariski dense subgroup. Hence, Theorem \ref{S arithmetic thm} fails in this setting. 

    We can apply the same reasoning above for when $\varphi$ has infinite order, but in this case, the Zariski closure is not irreducible. For example, consider the automorphism $\varphi \colon \Z^4 \to \Z^4$ given by
    $$
    \varphi \left( \begin{bmatrix}
        v_1 \\
        v_2 \\
        v_3\\
        v_4
    \end{bmatrix} \right)
    =
    \begin{bmatrix}
        2v_1 + v_2 \\
        v_1 + v_2\\
        v_4\\
        v_3
    \end{bmatrix}.
    $$
    We can view this automorphism as having the first two coordinates as length two vectors which are derived by multiplication by the matrix $A =  \begin{bmatrix}
        2 & 1\\
        1 & 1
    \end{bmatrix}$ and the last two coordinates swapped. We then have $\mathbf{M} = \overline{\left<\Psi_{\varphi}^\ell \: : \: \ell \in \Z  \right>}$ is isomorphic to a finite extension of $\Q$ in $\GL_4(\Q)$. Hence, $\Z^5$ is Zariski dense in $\Q^4 \rtimes \mathbf{M}$, but since $\Q^4 \rtimes \mathbf{M}$ is not irreducible, we cannot apply Theorem \ref{S arithmetic thm}.
\end{rem}

\begin{rem}
 We further verify the need of $\Gamma$ in Theorem $\ref{main_thm}$ to intersect the $\Q$-split maximal torus discretely in the following way. It is straightforward to see that $BS(1,n)$ admits a Zariski dense faithful representation in $\Q \rtimes \Q^*$. We also have that for any set of primes $\{p_1, p_2, \ldots, p_\ell\}$, the group $\Gamma$ given by
$$
\Gamma = \left< a, b_1, \cdots, b_\ell \: | \: b_iab_i^{-1} = a^{p_i}, b_i b_j = b_j b_i \right>
$$
is a Zariski dense subgroup of $\Q \rtimes \Q^*$. In particular, $\Gamma$ is given by
$$
\Gamma = \Z \left[\frac{1}{p_1}, \ldots, \frac{1}{p_\ell} \right] \rtimes \Z^\ell
$$
where the $i$-th coordinate of $\Z^\ell$ acts on $\Z[\frac{1}{p_1}, \ldots \frac{1}{p_\ell}]$ by multiplication by powers of $p_i$. However, it follows that $\Gamma$ has Hirsch length $\ell + 1$ and $BS(1,n)$ has Hirsch length $2$. Hence, $\Gamma$ and $BS(1,n)$ are Zariski dense in $\Q \rtimes \Q^*$ but do not have isomorphic cohomology rings over $\Q$.
\end{rem}


\subsection{Outline} We begin Section \ref{Section: Preliminaries} with notation and basic definitions on the rank of an abelian group and of solvable groups. We also provide standard definitions regarding linear algebraic groups including the Zariski closure, the unipotent radical, and tori as a means to introduce irreducible solvable linear algebraic groups. In Section \ref{Section: Cohom}, we develop rational cohomology and Lie algebra cohomology of a linear algebraic group defined over $\Q$ with an emphasis on their ring structure via the cup product. We also introduce background for spectral sequences in cohomology in Section 3. In Section \ref{Section: Unipotent}, we study cohomology of Zariski dense subgroups of irreducible unipotent $\Q$-defined linear algebraic groups and show that the restriction map induces an isomorphism of cohomology with arbitrary coefficients in this setting. We dedicate Section \ref{Section: Proof of main thms} to the proof of our main Theorems \ref{main_thm} and \ref{resmap_solv}. \newline

\paragraph{\textbf{Acknowledgments.}} The authors would like to thank Peter Kropholler for helpful conversations regarding this work and his comments on previous drafts which have significantly improved the results of this paper.
The second and third author were supported by the National Science Center Grant Maestro-13 UMO-2021/42/A/ST1/00306. The second author was also supported by Fondation Sciences Mathématiques de Paris. We thank the referee for suggestions and for pointing out mistakes in an earlier version of this paper.

\section{Preliminaries}\label{Section: Preliminaries}


\subsection{Solvable groups} We refer the reader to \cite[p. 83]{lennox_robinson, segal}.
Let $A$ be an abelian group. The \emph{torsion-free rank} of $A$ is defined as

\begin{center}
$r_0(A) = \dim_{\Q}(A \otimes_{\Z} \Q).$
\end{center}

Given a prime $p$, the $p$-\emph{rank} of $A$ is given by

\begin{center}
$r_p(A) = \dim_{\F_p}(\Hom(\F_p, A))$
\end{center}
where $\F_p$ is the field with $p$ elements. 

Let $G$ be a group. The \emph{lower central series} of $G$ is defined recursively where the first term is given by $\gamma_1(G) = G$ and the $i$-th term by $\gamma_i(G) = [G, \gamma_{i-1}(G)]$ for $i > 1$. We say that $G$ is a \emph{nilpotent group of step length $c$} if $c$ is the smallest natural number such that $\gamma_{c+1}(G) = \{1\}$. We say that $G$ is a nilpotent group if the step length is unspecified.

The \emph{derived series} of $G$ is defined recursively where the $0$-th term is given by $G^{(0)} = G$ and the $i$-th term by $G^{(i)} = [G^{(i-1)}, G^{(i-1)}]$ for $i > 0$. We say that $G$ is a \emph{solvable group of derived length $d$} if $d$ is the smallest natural number such that $G^{(d+1)} = \{1\}$. When the derived length is unspecified, we simply refer to $G$ as a solvable group. 

We denote the \emph{Fitting subgroup} of $G$ by 
\begin{center}
$\Fitt(G) = \left< N \: : \: N \trianglelefteq G, \: N \text{ nilpotent} \right>.$
\end{center}

\begin{defn}
A solvable group $G$ has \emph{finite abelian ranks} or is a \emph{solvable FAR-group} if it admits a subnormal series $\{G_i\}_{i=0}^k$ of finite length such that each successive quotient $G_{i} / G_{i-1}$ is abelian and $r_p(G_{i}/G_{i-1})$ is finite for $p=0$ or $p$ prime for each $i>0$. The torsion-free rank of $G$ is given by
\begin{center}
$h(G) = \displaystyle\sum_{i=0}^k r_0(G_{i} / G_{i-1}).$
\end{center}
This value is also known as the \emph{Hirsch length} of $G$.
\end{defn}

We say that a solvable group $G$ is \emph{polyrational} if it admits a subnormal series $\{H_{i}\}_{i=0}^k$ of finite length such that each successive quotient $H_{i}/H_{i-1} \leq \Q$.  We refer to the integer $k$ as the \emph{rational series length}. We note that if $G$ is polyrational with a rational series of length $k$, then $G$ has Hirsch length $k$. 


\subsection{Linear algebraic groups}\label{background_linear_algebraic_groups}
We refer the reader to \cite{borel, springer} for more background on linear algebraic groups.

\noindent Let $R$ be a subring of $\Q$. We denote the polynomial ring over $R$ in $\ell$ variables as $R[X_1, \ldots, X_\ell]$. Given a collection of polynomials $\{F_i\}_{i=1}^k$, we denote the variety in $R[X_1, \ldots, X_\ell]$ defined by the polynomials $\{F_i\}_{i=1}^k$ as $V(F_1, \ldots, F_k)$. 

A $\Q$-defined \emph{linear algebraic group} $\textbf{G} \leq \GL(n,\Q)$ is an algebraic variety 

\begin{center}
$\textbf{G} = V(F_1, \ldots, F_m) = \left\{ X \in \Q^{n^2} : F_i(X) = 0, i = 1, \ldots, m\right\}$
\end{center}

which is also a group such that the maps $\mu \colon \textbf{G} \times \textbf{G} \to \textbf{G}$ and $\iota \colon \textbf{G} \to \textbf{G}$ with $\mu(x,y) = xy$ and $i \colon x \to x^{-1}$ are morphisms of $\Q$-varieties. The group of $R$-points of $\textbf{G}$ is given by $\textbf{G}(R) = \textbf{G} \cap \GL(n,R)$. 

We say that $\textbf{G}$ is \emph{irreducible} if there exists polynomials $\mathcal{I}_{\mathbf{G}} = (F_1, \cdots, F_\rho) \subset \Q[x_{11}, \ldots, x_{nn}, \text{Det}^{-1}]$ such that the coordinate ring of $\textbf{G}$ is given by 
\begin{center}
$\Q[\textbf{G}] = \frac{\Q[x_{11}, \ldots, x_{nn}, \text{Det}^{-1}]}{\mathcal{I}_{\textbf{G}} \cap \Q[x_{11}, \ldots, x_{nn}, \text{Det}^{-1}]}$
\end{center}
where $\mathcal{I}_{\textbf{G}} \cap \Q[x_{11}, \ldots, x_{nn}, \text{Det}^{-1}]$ is a prime ideal. We note that for a linear algebraic group, irreducibility is equivalent to connectedness.

A $\Q$-\emph{closed subgroup} $\textbf{H}$ of the $\Q$-defined linear algebraic group $\textbf{G}$ is a subgroup that is closed in the Zariski topology. Let $\textbf{G}$ and $\textbf{H}$ be $\Q$-defined linear algebraic groups. A \emph{$\Q$-morphism of linear algebraic groups} $\varphi \colon \textbf{G} \to \textbf{H}$ is a group homomorphism which is also a morphism of $\Q$-varieties. We note that quotients are defined in the category of $\Q$-defined linear algebraic groups. In particular, if $\textbf{N}$ is a $\Q$-closed normal subgroup of $\textbf{G}$, then $\textbf{G} / \textbf{N}$ is a $\Q$-defined linear algebraic group where the natural projection $\textbf{G} \to \textbf{G} / \textbf{N}$ is a $\Q$-defined homomorphism.

Given a subgroup $H \leq \textbf{G}(\Q)$, we denote its Zariski closure as $\overline{H}$ which is the smallest $\Q$-defined linear algebraic group defined over $\Q$ which contains $H$. Given closed subgroups $X,Y \leq \textbf{G}$, if $X$ normalizes $Y$, then $\overline{X}$ normalizes $\overline{Y}$, and in particular, it follows that $\overline{[X,Y]} = [\overline{X}, \overline{Y}]$. Subsequently,  $\gamma_i(\overline{H}) = \overline{\gamma_i(H)}$ and $\overline{H^{(i)}} = (\overline{H})^{(i)}$. Thus, if $H$ is nilpotent of step length $c$, then $\overline{H}$ is nilpotent of step length $c$, and if $H$ is solvable of derived length $\ell$, then $\overline{H}$ is solvable of derived length $\ell$.

The \emph{unipotent radical} $\textbf{U}(\textbf{G})$ is defined as the set of unipotent elements of the maximal closed connected normal solvable subgroup of $\textbf{G}$ which is a Zariski closed subgroup defined over $\Q$. We are most concerned with linear algebraic groups that are virtually solvable, and in this context, the subgroup $\textbf{U}(\textbf{G})$ is the set of unipotent elements of $\textbf{G}$. Any surjective $\Q$-defined homomorphism $\varphi \colon \textbf{G} \to \textbf{H}$ of $\Q$-defined groups maps unipotent elements to unipotent elements, and subsequently, $\varphi(\textbf{U}(\textbf{G})) = \textbf{U}(\textbf{H})$. 

We denote $\G^\Q_a = \Q$ to be $\Q$ viewed as a linear algebraic group with its additive group structure and $\G_m^\Q = \Q^*$ with its multiplicative structure. Then, a \emph{torus} is a linear algebraic group which is isomorphic to a closed subgroup of $(\G_m^\Q)^k$ for some $k$.

We are most concerned when $\G$ is an irreducible solvable $\Q$-defined linear algebraic group. It follows that $\G$ can be written as the semidirect product 

\begin{center}
$\G = \U \rtimes \T$
\end{center}

where $\U$ is the $\Q$-defined unipotent radical and $\T$ is a $\Q$-split maximal torus.



\section{Rational cohomology and Lie algebra cohomology of linear algebraic groups defined over $\Q$} \label{Section: Cohom}

We dedicate this section to the rational cohomology of $\Q$-defined linear algebraic groups and Lie algebra cohomology of $\Q$-defined Lie algebras.

\subsection{Cohomology of rational $\G$-modules}
Let $\G$ be a $\Q$-defined linear algebraic group. For a finite dimensional $\Q$-vector space $V$, we say that $V$ is a \emph{rational} $\G$-\emph{module} if it admits a $\Q[\textbf{G}]$-comodule structure via the $\Q$-linear map $\delta_V \colon V \to V \otimes_{\Q} \Q[\textbf{G}]$. It follows that we have a cochain complex $C^*(\Q[\textbf{G}], V)$ defined by
\begin{center}
$C^n(\Q[\textbf{G}],V) = V \otimes (\bigotimes_{i=1}^n \Q[\textbf{G}])$,
\end{center}
with codifferential
\begin{center}
$\partial^n = \displaystyle\sum_{i=0}^{n+1}(-1)^i \partial_i^n$
\end{center}

where
\begin{eqnarray*}
\partial_0^n(v \otimes f_1 \otimes \cdots \otimes f_n) &=& \delta_V(v) \otimes f_1 \cdots \otimes f_n,\\
\partial_i^n(v \otimes f_1  \otimes \cdots \otimes f_n) &=& v \otimes f_1 \otimes \cdots \otimes \delta f_i \otimes  \cdots \otimes f_n, \quad 1 \leq i \leq n,\\
\partial_{n+1}^n(v \otimes f_1 \otimes \cdots \otimes f_n) &=&  v \otimes f_1 \otimes \cdots \otimes f_n \otimes 1.
\end{eqnarray*}
The \emph{rational cohomology groups} of $\textbf{G}$ with coefficients in $V$ are then the cohomology groups of the complex $C^*(\Q[\textbf{G}],V):$

\begin{center}
$H^n(\textbf{G},V) = \ker (\partial^{n}) / \text{Im} (\partial^{n-1})$
\end{center}
for $n \geq 1$.
For $n = 0$, we define $H^0(\textbf{G},V) = \ker (\partial^{0})$; this is the subspace of all $v \in V$ such that $\Delta_V(v) = v \otimes 1,$ i.e.

\begin{center}
$H^0(\textbf{G},V) = V^{\textbf{G}}.$
\end{center}

We note that the groups $H^n(\textbf{G},V)$ are the derived functors of the functor $V \to V^{\textbf{G}}$ in the category of $\Q[\textbf{G}]$-comodules.

When $V = \Q$ is the trivial $\Q[\G]$-comodule, we have

\begin{center}
$C^0(\Q[\textbf{G}], \Q) = \Q \quad \text{ and } \quad C^n(\Q[\textbf{G}], \Q) = \Q[\textbf{G}]^{\otimes_n}$
\end{center}

for $n > 0$. For a finite dimensional rational $\textbf{G}$-module $M$, the cup product

\begin{center}
$H^i(\textbf{G}, M) \otimes H^j(\textbf{G}, \Q) \overset{\cup}{\longrightarrow} H^{i+j}(\textbf{G}, M)$
\end{center}

is defined on the cochain level as follows. For 
\begin{center}
    $\varphi = v \otimes f_1 \otimes \cdots \otimes f_i \in C^i(\Q[\textbf{G}], M)$ 
    \end{center}
    and 
    \begin{center}
    $\psi  =g_1  \otimes \cdots \otimes g_j \in C^j(\Q[\textbf{G}], \Q)$, 
\end{center}
we have
\begin{center}
$\varphi \cup \psi = \varphi \otimes \psi = v \otimes f_1 \otimes \cdots \otimes f_i \otimes g_1 \otimes \cdots \otimes g_j \in C^{i+j}(\Q[\textbf{G}], M).$
\end{center}

The total differential is given by

\begin{center}
$d(\varphi \otimes \psi) = d\varphi \otimes \psi + (-1)^{\text{deg}(\varphi)}\varphi \otimes d\psi$,
\end{center}

and it can be easily checked that $\cup$ defines a map of complexes

\begin{center}
$C^*(\Q[\textbf{G}], M) \otimes C^*(\Q[\textbf{G}], \Q)  \overset{\cup}{\longrightarrow} C^*(\Q[\textbf{G}], M).$ 
\end{center}

By the K\"{u}nneth theorem, we then have the cohomology pairing

\begin{center}
$H^i(\textbf{G}, M) \otimes H^j(\textbf{G}, \Q) \overset{\cup}\longrightarrow H^{i+j}(\textbf{G}, M).$
\end{center}

Let $\Gamma \leq \textbf{G}(\Q)$ be a subgroup of the $\Q$-points of $\G$. Given the inclusion $\Gamma \to \G(\Q)$ and a finite dimensional rational $\G$-module $M$, the \emph{restriction map} of cochain complexes

\begin{center}
$r_M \colon C^*(\Q[\textbf{G}], M) \longrightarrow C^*(\Gamma, M)$
\end{center}

to be defined as follows. For $\varphi = v \otimes f_1 \otimes \cdots \otimes f_n \in C^n(\Q[\textbf{G}], M),$ we define

\begin{center}
$r_M(\varphi)(\gamma_1, \cdots, \gamma_n) = f_1(\gamma_1) \cdots f_n(\gamma_n) v \in M.$
\end{center}

It follows that this map of complexes induces a restriction map of cohomology groups

\begin{center}
$r_M: H^n(\textbf{G},M) \to H^n(\Gamma,M).$
\end{center}

We also have that the cup product $H^i(\Gamma, M) \otimes H^j(\Gamma, \Q) \overset{\cup}{\longrightarrow} H^{i+j}(\Gamma, M)$ is defined on the cochain level in the following way. If $\varphi \in C^i(\Gamma, M)$ and $\psi \in C^j(\Gamma, \Q)$, then

\begin{center}
$\varphi \cup \psi (\gamma_1, \ldots, \gamma_i, \gamma_{i+1}, \ldots, \gamma_{i+j}) = \varphi(\gamma_1, \ldots, \gamma_i) \cdot \psi(\gamma_{i+1}, \ldots, \gamma_{i+j}).$
\end{center}

An immediate consequence of the above definitions is that the restriction map $r_M$ preserves cup products such that the following diagram commutes:

\begin{equation}\label{D1}
\begin{tikzcd}
H^i(\G, M) \otimes H^j(\G, \Q) \ar[d, "r_M \otimes r_\Q"] \ar[r, "\cup"] & H^{i+j}(\G,M) \ar[d, "r_M"]  \\
H^i(\Gamma, M) \otimes H^j(\Gamma, \Q) \ar[r, "\cup"]  & H^{i+j}(\Gamma, M)
\end{tikzcd}.
\end{equation}

\subsection{Cohomology of $\mathfrak{g}_{\Q}$-modules}
Given a finite dimensional Lie algebra $\mathfrak{g}_\Q$ over the rationals $\Q$, a finite dimensional $\Q$-vector space $V$ is a $\mathfrak{g}_\Q$-module if there exists a Lie algebra homomorphism $\mathfrak{g}_\Q \to \text{End}_\Q(V)$ where $\text{End}_\Q(V)$ is the Lie algebra of $\Q$-defined endomorphisms of $V$ with Lie bracket given by $[A,B] = AB - BA$. If $M$ is a finite dimensional rational $\textbf{G}$-module, we can consider $M$ as a $\mathfrak{g}_{\Q}$-module, where $\mathfrak{g}_{\Q}$ is the $\Q$-defined Lie algebra of $\textbf{G}$, by taking the differential of the associated representation. We will only consider $\mathfrak{g}_{\Q}$-modules derived from rational representations.

Given an ideal $\mathfrak{h}$ in $\mathfrak{g}_{\Q}$, we denote $C^n(\mathfrak{h},M)$ as the vector space of $n$-linear alternating functions on $\mathfrak{h}$ with values in $M$. We give a $\mathfrak{g}_{\Q}$-module structure on $C^n(\mathfrak{h},M)$ in the following way. 

For $n=0$, $C^0(\mathfrak{h},M) = M$ is trivially a $\mathfrak{g}_{\Q}$-module. For $n>0$, $\varphi \in C^n(\mathfrak{h},M)$ and $X \in \mathfrak{g}_{\Q}$, we define

\begin{center}
$(X \cdot \varphi)(Y_1, \ldots, Y_n) = X \cdot \varphi(Y_1, \ldots, Y_n) - \displaystyle\sum_{i=1}^n \varphi(Y_1, \ldots, Y_{i-1}, [X,Y_i], Y_{i+1}, \ldots, Y_n).$
\end{center}

The differential in $C^*(\mathfrak{h},M)$ is a $\mathfrak{g}_{\Q}$-module map such that

\begin{center}
$d(X \cdot \varphi) = X \cdot (d\varphi).$
\end{center}

It follows that the $\mathfrak{g}_{\Q}$-action on cochains induces a $\mathfrak{g}_{\Q}$-module structure on the cohomology groups $H^i(\mathfrak{h},M)$. It then follows that $\mathfrak{h}$ admits a trivial action on every cohomology group $H^n(\mathfrak{h},M).$ 

Given $\mathfrak{g}_{\Q}$-modules $M,N,$ and $P$, a \emph{pairing} from $M$ and $N$ to $P$ is a $\Q$-bilinear map $(m,n) \to m \cup n$ of $M \times N \to P$ such that for all $X \in \mathfrak{g}_{\Q}$, we have

\begin{center}
$X \cdot (m \cup n) = (X \cdot m) \cup n = m \cup (X \cdot n).$
\end{center}

For an ideal $\mathfrak{h}$ in $\mathfrak{g}_{\Q}$, we will define a pairing of cohains

\begin{center}
$C^i(\mathfrak{h},M) \otimes C^j(\mathfrak{h},N) \overset{\cup}{\longrightarrow} C^{i+j}(\mathfrak{h},P)$
\end{center}

as follows. For $\varphi \in C^i(\mathfrak{h},M)$ and $\psi \in C^j(\mathfrak{h},N),$ we have

\begin{center}
$\varphi \cup \psi(Y_1, \ldots, Y_{i+j}) = \displaystyle\sum_{\sigma \text{ is a } (i,j)- \text{shuffle}} (-1)^{\text{sign}(\sigma)} \varphi(Y_{\sigma(1)}, \ldots, Y_{\sigma(i)}) \cup \psi(Y_{\sigma(i+1)}, \ldots, Y_{\sigma(i+j)})$
\end{center}

where a permutation $\sigma \in \text{Sym}(i+j)$ is an $(i,j)$-shuffle if $\sigma(1) < \sigma(2) < \cdots \sigma(i)$ and $\sigma(i+1) < \sigma(i+2) < \cdots \sigma(i+j).$ We then have

\begin{center}
$X \cdot ( \varphi \cup \psi) = (X \cdot \varphi) \cup \psi + \varphi \cup (X \cdot \psi)$
\end{center}

for all $X \in \mathfrak{g}_{\Q}.$ Additionally,

\begin{center}
$d(\varphi \cup \psi) = d \varphi \cup \psi (-1)^{\text{Deg}(\varphi)}\varphi \cup d\psi.$
\end{center}

It then follows that the pairing $C^i(\mathfrak{h},M) \otimes C^j(\mathfrak{h},N) \overset{\cup}{\longrightarrow} C^{i+j}(\mathfrak{h},P)$ induces a cohomology pairing 
    
 \begin{center}
$H^i(\mathfrak{h},M) \otimes H^j(\mathfrak{h},N) \overset{\cup}{\longrightarrow} H^{i+j}(\mathfrak{h},P).$
   \end{center}
   
By properties of the cup product, there exists an induced pairing on the $\mathfrak{g}_{\Q}$-annihilated elements

\begin{center}
$H^i(\mathfrak{h},M)^{\mathfrak{g}_{\Q}} \otimes H^j(\mathfrak{h},N)^{\mathfrak{g}_{\Q}} \overset{\cup}{\longrightarrow} H^{i+j}(\mathfrak{h},P)^{\mathfrak{g}_{\Q}}.$
\end{center}

If $N = \Q$ is the trivial $\mathfrak{g}_{\Q}$-module with pairing $M \times \Q \overset{\cup}{\rightarrow}  M$ given by $(v, \alpha) \to \alpha v,$ we then have the usual cup product:

\begin{center}
$H^i(\mathfrak{h},M) \otimes H^j(\mathfrak{h},\Q) \overset{\cup}{\longrightarrow} H^{i+j}(\mathfrak{h},M).$
\end{center}

By \cite[Theorem 13]{Lie_cohomoloby}, we have the following theorem.
\begin{thm}{\label{T31}}
    Let $\mathfrak{g}_\Q$ be a finite dimensional Lie algebra over $\Q$, and let $M$ be a finite dimensional $\mathfrak{g}_\Q$--module. Suppose that $\mathfrak{u}$ is an ideal of $\mathfrak{g}_\Q$ such that $\mathfrak{t} = \mathfrak{g} / \mathfrak{u}$ is semisimple. Then for all $n \geq 0$, we have the following isomorphism
    
    \begin{center}
    $H^n(\mathfrak{g}_\Q,M) \cong \displaystyle\bigoplus_{i+j= n} H^i(\mathfrak{u}, M)^{\mathfrak{g}_\Q} \otimes H^j(\mathfrak{t}, \Q)$
    \end{center}
    
which is multiplicative for paired modules.
\end{thm}

The multiplicative structure of the isomorphism of Theorem \ref{T31} can be viewed as follows.

A given pairing of $\mathfrak{g}$-modules $M \times N \overset{\cup}{\rightarrow} P$ induces a cohomology pairing 

\begin{center}
$H^i(\mathfrak{g},M) \otimes H^j(\mathfrak{g},N) \overset{\cup}{\longrightarrow} H^{i+j}(\mathfrak{g},P), \quad i,j \geq 0.$
\end{center}

By Theorem \ref{T31}, we then have 
\begin{eqnarray*}
H^i(\mathfrak{g}_\Q,M) &\cong& \bigoplus_{k+\ell = i} H^k(\mathfrak{u},M)^{\mathfrak{g}_\Q} \otimes H^\ell(\mathfrak{t},\Q),\\
H^j(\mathfrak{g}_\Q,N) &\cong& \bigoplus_{p+q = j} H^p(\mathfrak{u},N)^{\mathfrak{g}_\Q} \otimes H^q(\mathfrak{t},\Q).
\end{eqnarray*}

Given $u_1 \otimes v_1 \in H^i(\mathfrak{g}_\Q,M)$ and $u_2 \otimes v_2 \in H^j(\mathfrak{g},N),$ we have that

\begin{center}
$(u_1 \otimes v_1) \cup (u_2 \otimes v_2) = (-1)^{\ell p}(u_1 \cup_1 u_2) \otimes (v_1 \cup_2 v_2) \in H^{i+j}(\mathfrak{g}_\Q,P)$
\end{center}

where
\begin{center}
   $ H^*(\mathfrak{t},\Q) \otimes H^*(\mathfrak{t}, \Q) \overset{\cup_1}{\longrightarrow} H^*(\mathfrak{t}, \Q)$,
\end{center}
and
\begin{center}
$H^*(\mathfrak{u},M)^{\mathfrak{g}_\Q} \otimes H^*(\mathfrak{u},N)^{\mathfrak{g}_\Q} \overset{\cup_2}{\longrightarrow} H^*(\mathfrak{u},P)^{\mathfrak{g}_\Q}$
\end{center}

are the other pairings used. Therefore, there exists a decomposition of the multiplicative structure in 
$H^*(\mathfrak{g}_\Q,\Q)$ where $\Q$ is a trivial $\mathfrak{g}_\Q$-module. If $\alpha \in H^i(\mathfrak{t},\Q)$ such that 

\begin{center}
$\alpha \otimes 1 \in H^i(\mathfrak{t},k) \otimes H^0(\mathfrak{u},\Q)^{\mathfrak{g}} = H^i(\mathfrak{t},\Q) \otimes \Q$
\end{center}

and $\beta \in H^j(\mathfrak{u},\Q)^{\mathfrak{g}_\Q}$ such that

\begin{center}
$\beta \otimes 1 \in H^0(\mathfrak{u},\Q)^{\mathfrak{g}_\Q} \otimes H^j(\mathfrak{u},\Q)^{\mathfrak{g}_\Q},$
\end{center}
then for the pairing $\Q \times \Q \overset{\cup}{\longrightarrow} \Q$ induced by multiplication, we have

\begin{center}
$(\alpha \otimes 1) \cup (1 \otimes \beta) = \alpha \otimes \beta \in H^{i+j}(\mathfrak{g},\Q).$
\end{center}

\subsection{Cohomology spectral sequences} \label{Section: Spectral sequences}
We provide the necessary background on cohomology spectral sequences and refer the reader to \cite{Weibel}. We begin with the following definition.

\begin{defn}
A cohomology spectral sequence in an abelian category $\mathcal{A}$  is given by ${(E_r^{pq}, d_r^{pq})}_{r \geq 0}$ where each $E_r^{pq}$ is an object of $\mathcal{A}$ and $d_r^{pq}: E_r^{pq} \to E_r^{p+r, q-r+1}$ are morphisms in $\mathcal{A}$ satisfying $d_r^{p+r, q-r+1} \circ d_r^{pq} =0$ and where $$E_{r+1}^{pq} = \textrm{Ker}(d_r^{pq})/\textrm{Im}(d_r^{p-r, q+r-1}).$$
\end{defn}
We are most concerned with first quadrant cohomology spectral sequences which are those in which $E_r^{pq} = 0$ unless $p \geq 0$, $q \geq 0$. In particular, we will be considering pairs of first quadrant cohomology spectral sequences $E, E'$ such that there exists a morphism between them. We have the following definition.

\begin{defn}A morphism of cohomology spectral sequences $f: E \to E'$ is a family of maps $f_r^{pq}: E_r^{pq} \to {E'_r}^{pq}$ in $\mathcal{A}$ with $(d')_r^{pq} \circ f_r^{pq} = f_r^{p+r, q-r+1} \circ d_r^{pq}$ such that each $f_{r+1}^{pq}$ is the map induced by $f_r^{pq}$.
\end{defn}

We have the following lemma which is a consequence of the 5-lemma.

\begin{lemma}{(Mapping Lemma, \cite{Weibel}, p. 123)}
Let $f \colon E \to E'$ be a morphism of spectral sequences such that for some fixed $r$, $f_r^{pq} \colon E_r^{pq} \to {E'_r}^{pq}$ is an isomorphism for all $p$ and $q$. Then $f_s^{pq} \colon E_s^{pq} \to {E'_s}^{pq}$ is an isomorphism for all $s \geq r.$
\end{lemma}

Additionally, we say that $E_r^{pq}$ admits a $\emph{multiplicative structure}$ if we have a morphism of spectral sequences 

\begin{center}
$E_r^{pq} \times E_r^{p'q'} \to E_r^{p+p',q+q'}$ 
\end{center}

such that the differentials $d_r$ satisfy

\begin{center}
$d_r(xy) = d_r(x)y + (-1)^pxd_r(y)$ 
\end{center}

for $x \in E_r^{pq}, y \in E_r^{p'q'}$. We note that if there is a multiplicative structure for a given $r$, it induces multiplicative structures on $E_s^{pq}$ for $s \geq r$. 

We now turn to the way in which cohomology spectral sequences arise. We begin with a cochain complex $C$. A $\emph{filtration}$ on $C$ is an ordered family of chain subcomplexes 

\begin{center}
$... \subseteq F^{p-1}(C) \subseteq F^p(C) \subseteq ...$
\end{center}

which give rise to a spectral sequence beginning with $E_0^{pq} = F^p(C^{p+q})/F^{p+1}(C^{p+q})$ and $E_1^{pq} = H^{p+q}(E_0^{p*})$. 

We further have that a filtration on $C$ induces a filtration on the cohomology of $C$ such that $F^p(H^n(C))$ is the image of the map $H^n(F^p(C)) \to H^n(C)$:

\begin{center}
$... \subseteq F^{p+1}(H^n) \subseteq F^p(H^n) \subseteq F^{p-1}(H^n) \subseteq ...$
\end{center}

We say a filtration on $C$ is \emph{bounded} if for each $n$, there are integers $s < t$ such that $F_s(C_n) = 0$ and $F_t(C_n) = C_n$, and $\emph{canonically bounded}$ if $F_{-1}(C)=0$ and $F_n(C_n)=C_n$ for each $n$. If the filtration on $C$ is bounded, then the induced filtration on $H^n$ is bounded and the spectral sequence is bounded. In this case, the spectral sequence admits the following convergence properties. 

\begin{defn}
We say that a bounded cohomology spectral sequence $\emph{converges}$ to $H^*$ if we are given a family of objects $H^n$ of $\mathcal{A}$ such that each $H^n$ admits a finite filtration 
 
\begin{center}
$0 = F^s(H^n) \subseteq \ldots \subseteq F^{p+1}(H^n) \subseteq F^p(H^n) \subseteq F^{p-1}(H^n) \subseteq \ldots \subseteq F^t(H^n) = H^n$;
\end{center}
 Furthermore, the morphisms $\beta^{pq}: E_{\infty}^{pq} \to F^p(H^{p+q}(C))/F^{p+1}(H^{p+q}(C))$ are isomorphisms. 
\end{defn}

We note that every canonically bounded filtration on $C$ gives rise to a first quadrant spectral sequence converging to $H^*$. Hence, it follows that $H^n$ has a finite filtration

\begin{center}
$0 = F^{n+1}(H^n) \subseteq F^n(H^n) \subseteq \ldots \subseteq F^{1}(H^n) \subseteq F^0(H^n) = H^n$,
\end{center}

and again, we have that $\beta^{pq}: E_{\infty}^{pq} \to F^p(H^{p+q}(C))/F^{p+1}(H^{p+q}(C))$ are isomorphisms. 

If $\{E_r^{pq}\}$ and $\{{E'_r}^{pq}\}$ converge to $H^*$ and $H'^{*}$, respectively, we say that a map $h: H^* \to H'^*$ is \emph{compatible} with a morphism $f: E \to E'$ of spectral sequences if $h$ maps $F^p(H^n) \to F^p(H'^n)$ and the associated maps 
$$
F^p(H^n)/F^{p+1}(H^n) \to F^p(H'^n)/F^{p+1}(H'^n)
$$ correspond under $\beta$ and $\beta'$ to $f_{\infty}^{pq}: E_{\infty}^{pq} \to {E'_{\infty}}^{pq} (q=n-p)$.

We have the following theorem. 

\begin{thm}{(Classical Convergence Theorem, \cite{Weibel}, p. 135)}{\label{convthm}} Suppose that the filtration on $C$ is bounded. Then, the spectral sequence is bounded and converges to $H^*(C)$:

\begin{center}
$E_1^{pq} = H^{p+q}(F^p(C)/F^{p+1}(C)) \Rightarrow H^{p+q}(C)$.
\end{center}

Moreover, the convergence is natural in the sense that if $f: C \to C'$ is a map of filtered complexes, then the map $f^*: H^*(C) \to H^*(C')$ is compatible with the corresponding map of spectral sequences. 
\end{thm} 

\subsubsection{Spectral sequences in group and Lie algebraic cohomology} 

We now discuss the spectral sequences that converge to $H^*(\mathfrak{g}_\Q, M)$ and $H^*(\Gamma, M).$
We begin with the following spectral sequence in group cohomology. 

\begin{thm}{(Lyndon-Hochschild-Serre, \cite{Weibel}, p. 195)} Let $G$ be a group with a $G$-module $M$. For any group extension

\begin{center}
$1 \to H \to G \to G/H \to 1$,
\end{center}
there is a first quadrant spectral sequence given by

\begin{center}
$E_2^{pq} = H^p(G/H, H^q(H,M)) \Rightarrow H^{p+q}(G,M)$.
\end{center}
\end{thm}

We have an analogous spectral sequence in Lie algebra cohomology. 

\begin{thm}{(Hochschild-Serre, \cite{Weibel}, p. 232)}
Given a Lie algebra $\mathfrak{g}$, an ideal $\mathfrak{h} \subset \mathfrak{g}$,  and a $\mathfrak{g}$-module $M$, for any short exact sequence of Lie algebras

\begin{center}
$1 \to \mathfrak{h} \to \mathfrak{g} \to \mathfrak{g}/\mathfrak{h} \to 1$,
\end{center}
there is a first quadrant spectral sequence given by 

\begin{center}
$E_2^{pq} = H^p(\mathfrak{g}/\mathfrak{h}, H^q(\mathfrak{h},M)) \Rightarrow H^{p+q}(\mathfrak{g},M).$
\end{center}
\end{thm}


\section{Cohomology of $\mathbb{Q}$-defined unipotent linear algebraic groups with coefficients}\label{Section: Unipotent}

The goal of this section is to show Proposition \ref{unipotent_2}, which can be thought as Theorem \ref{resmap_solv} in the particular case of a $\Q$-defined unipotent linear algebraic group. This will require first to understand the restriction map for trivial coefficients, and then to study cohomology of rational modules without invariant vectors.

\subsection{Cohomology with trivial coefficients}
We start this subsection by showing some structural results regarding lattices in unipotent $\Q$-defined linear algebraic groups. We have the following basic lemma. 

\begin{lemma}\label{lem:polyrational}
    Let $\U$ be an irreducible unipotent $\Q$-defined linear algebraic group and $\Delta \leq \U(\Q)$ be a Zariski dense subgroup. Then $\Delta$ admits a normal polyrational series of length equal to $\dim_\Q(\U).$
\end{lemma}
\begin{proof}
    We proceed by induction on $\dim_\Q(\U)$. When $\dim_\Q(\U) = 1$, it follows that $\Delta$ is a non-trivial subgroup of $\Q$, and hence, it is polyrational by definition from which the result follows. Now suppose that $\dim_\Q(\U) > 1.$ We note that $Z(\U)$ is non-trivial, and thus, there exists a $\G_a^\Q \leq Z(\U).$ We then have that $\Delta / (\G_a^\Q(\Q) \cap \Delta)$ is Zariski dense in $\U/\G_a^\Q.$ Therefore, our inductive hypothesis implies that $\Delta / (\G_a^\Q(\Q) \cap \Delta)$ is polyrational. Since $\G_a^\Q(\Q) \cap \Delta \leq \Delta$ and $\Delta$ is a polyrational extension of a non-trivial subgroup of $\Q$, it must be polyrational. Additionally,  the inductive case implies that $\Delta / (\G_a^\Q(\Q) \cap \Delta)$ admits a normal polyrational series $\{\overline{\Delta_i}\}_{i=2}^{\dim_\Q(\U)}$. Then, letting $\pi \colon \Delta \to \Delta / (\G_a^\Q (\Q) \cap \Delta)$ be the natural projection, it follows that $\{\G_a^\Q(\Q) \cap \Delta\} \cup \{\pi^{-1}(\overline{\Delta_i})\}_{i=2}^{\dim_\Q(\U)}$ is a normal polyrational series.
\end{proof}

We also have the following lemma as an application of Lemma \ref{lem:polyrational}. 

\begin{lemma}\label{lem:cocompat_subgroup}
    Let $\U$ be an irreducible unipotent $\Q$-defined linear algebraic group and $\Delta \leq \U(\Q)$ be a Zariski dense subgroup. Then, there exists a torsion-free, finitely generated nilpotent subgroup $\Lambda \leq \Delta$ such that $h(\Lambda) = \dim_{\Q}(\U)$. 
\end{lemma}
\begin{proof}
    We proceed by induction on $\dim_\Q(\U)$, and for the base case, we have $\U = \Q$. Hence, $\Delta$ must have an element $x$ of infinite order, and thus, $\left<x\right> \cong \Z.$ Therefore, we let $\Lambda = \left<x\right>$. For the inductive case, suppose that $\dim_\Q(\U) > 1.$ By Lemma \ref{lem:polyrational}, there exists a normal polyrational series $\{\Delta_i\}_{i=1}^{\dim_\Q(\U)}$ for $\Delta$. It follows that $\Delta_{\dim_\Q(\U) - 1}$ is Zariski dense in an irreducible unipotent $\Q$-defined linear algebraic group of dimension $\dim_\Q(\U) - 1.$ We assume by induction that there exists a torsion-free, finitely generated nilpotent subgroup $\Lambda^\prime \leq \Delta_{\dim_\Q(\U) - 1}$ such that $h(\Lambda^\prime) = \dim_\Q(\U) - 1.$ We also note that $\Delta / \Delta_{\dim_\Q(\U) - 1}$ is Zariski dense in $\U/ \overline{\Delta_{\dim_\Q(\U) - 1}}$. Letting $\pi \colon \Delta \to \Delta / \Delta_{\dim_\Q(\U)}$ be the natural projection, the base case implies that there exists an element $x \in \Delta$ such that $\overline{\left<\pi(x)\right>} = \U/ \overline{\Delta_{\dim_\Q(\U) - 1}}$. We claim the group $\Lambda = \left<\Lambda^\prime, x\right>$ is our desired group. Since $\Delta_{\dim_\Q(\U) - 1} \cap \Lambda$ is normal in $\Lambda$, we have the following short exact sequence:
    \begin{center}
    $1 \longrightarrow \Delta_{\dim_\Q(\U) - 1} \cap \Lambda \longrightarrow \Lambda \longrightarrow \Lambda / (\Delta_{\dim_\Q(\U) - 1} \cap \Lambda) \longrightarrow 1.$
    \end{center}
    Since 
    \begin{center}
        $\Lambda^\prime \leq \Delta_{\dim_\Q(\U) - 1} \cap \Lambda \leq \Delta_{\dim_\Q(\U) - 1}$,
    \end{center} it follows that
    \begin{center} 
    $h(\Delta_{\dim_\Q(\U) - 1} \cap \Lambda) = \dim_\Q(\U) - 1.$
    \end{center}
    Since $\Lambda / (\Delta_{\dim_\Q(\U) - 1} \cap \Lambda) \cong \left<\pi(x)\right>$, it follows that $h(\Lambda) = \dim_\Q(\U).$ We end by noting that since $\U$ is irreducible and $\dim_\Q(\overline{\Lambda}) = \dim_\Q(\U)$, we must have that
    \begin{center}
$\U = \U^\circ \leq \overline{\Lambda} \leq \overline{\Delta} = \U$.
\end{center}
Hence, $\Lambda$ is Zariski dense in $\U.$
\end{proof}

In \cite[Lemma 2.21]{Kunkel}, Kunkel showed that the restriction map from an irreducible $\Q$-defined unipotent linear algebraic group to a discrete cocompact subgroup (in the Euclidean topology of the real points) with trivial $\Q$-coefficients is an isomorphism. The following proposition is an extension of this result to the non-discrete setting.

\begin{prop}\label{prop:zariski dense cohomology trivial coefficients}
    Let $\U$ be an irreducible unipotent $\Q$-defined linear algebraic group and $\Delta \leq \U(\Q)$ be a Zariski dense subgroup of finite abelian ranks. Then the restriction map $r_\Q \colon H^*(\U, \Q) \to H^*(\Delta, \Q)$ is an isomorphism of cohomology rings. 
\end{prop}
\begin{proof}
    If $\Delta$ is finitely generated, then this lemma follows from \cite[Lemma 2.21]{Kunkel}. We thus assume that $\Delta$ is not finitely generated. By Lemma \ref{lem:cocompat_subgroup}, there exists a torsion-free, finitely generated nilpotent subgroup $\Lambda \leq \Delta$ such that $h(\Lambda) = \dim_\Q(\U)$ and $\overline{\Lambda} = \U.$ Since $\Delta$ is countable, we may choose an enumeration $\{x_i\}_{i=1}^\infty$ of $\Delta$. Let $\Lambda_0 = \Lambda$ and $\Lambda_i = \left<\Lambda, x_1, \ldots, x_i\right>$ for $i>0$, and let $\varphi_i \colon \Lambda_i \to \Lambda_{i+1}$ be the natural inclusion.  We then have that $\underset{\longrightarrow} \lim (\Lambda_i, \varphi_i) \cong \Delta$. Each $\Lambda_i$ is a torsion-free, finitely generated, nilpotent group which contains a Zariski dense subgroup. Hence, each $\Lambda_i$ is Zariski dense in $\U$. Moreover, since $h(\Lambda_i) = \dim_\Q(\U)$, we must have that $\Lambda_i$ is a finite index subgroup of $\Lambda_{i+1}$.  Thus, $\Lambda_i$ is a Zariski dense, torsion-free, finitely generated subgroup of $\U(\R)$. Hence, \cite[Theorem 2.1]{raghunathan} implies that $\U(\R)/\Lambda_i$ is compact, and since $\Lambda_i$ is a discrete subgroup of $\U(\R)$ and $\U(\R)$ is contractible, the quotient $\U(\R)/\Lambda_i$ is an Eilenberg-Maclane space for $\Lambda_i.$ Then, each homomorphism $\varphi_i \colon \Lambda_i \to \Lambda_{i+1}$ has an associated unique pointed homotopy class of continuous maps $\psi_i \colon \U(\R)/\Lambda_i \to \U(\R)/\Lambda_{i+1}$ such that $\psi_i$ induces the map $\varphi_i$ of fundamental groups.  We now consider the following mapping telescope:
\begin{center}
    $T \quad \colon  \quad \U(\R) / \Lambda_0 \overset{\psi_0}\longrightarrow \U(\R) / \Lambda_1 \overset{\psi_1}\longrightarrow  \U(\R) / \Lambda_2 \overset{\psi_2}\longrightarrow \cdots$.
\end{center}
It follows that 
$$
\pi_1(T) = \underset{\longrightarrow} \lim (\Lambda_i, \varphi_i) = \Delta
$$  and for $\ell>1$, $\pi_\ell(T) = \underset{\longrightarrow}\lim (\pi_\ell(\U(\R)/\Lambda_i), (\psi_i)_*)) = 0$. Therefore, $T$ is an Eilenberg-Maclane space for $\Delta$. The inclusion map $\varphi_i \colon \Lambda_i \to \Lambda_{i+1}$ induces the restriction map $(\varphi_i)^* \colon H^*(\Lambda_{i+1}, \Q) \to H^*(\Lambda_i, \Q)$, and if $\theta_i \colon H^*(\U, \Q) \to H^*(\Lambda_i, \Q)$ is the restriction map, we have $(\varphi_i)^* \circ \theta_{i+1} = \theta_i$. \cite[Lemma 2.21]{Kunkel} implies that $\theta_i$ is an isomorphism of cohomology rings for all $i$, and therefore, $(\varphi_i)^*$ is an isomorphism. Hence, 
\begin{center}
$H^\ell(\Delta, \Q) = H^\ell(T,\Q)  = \underset{\longleftarrow} \lim (H^\ell(\U(\R)/\Lambda_i, \Q), (\psi_i)^*) = H^\ell(\Lambda, \Q) = H^\ell(\U,\Q). 
$
\end{center}
by \cite[Theorem 3F.5]{hatcher} and \cite[Lemma 2.21]{Kunkel}. Since $H^*(\U, \Q)$ and $H^*(\Delta, \Q)$ are finite dimensional over $\Q$ and abstractly isomorphic from the above arguments, they have equal dimension as $\Q$-vector spaces. Therefore, it remains to show that $r_\Q: H^*(\U, \Q) \to H^*(\Delta, \Q)$ is an injection. 

We have the following commutative diagram of restriction maps:

\begin{equation}
\begin{tikzcd}
    H^*(\U, \Q) \arrow[rr, "r_\Q"] \arrow[rrdd, "r_{\Q}''"'] && H^*(\Delta, \Q) \arrow[dd, "r_\Q'"] \\
    \\
    {} && H^*(\Lambda, \Q) 
\end{tikzcd}.
\end{equation}

By \cite{Kunkel}, it follows that $r_\Q''$ is an isomorphism of cohomology rings, and therefore, $r_\Q$ is an injection. This implies that $r_\Q: H^*(\U, \Q) \to H^*(\Delta, \Q)$ is an isomorphism of cohomology rings. 
\end{proof}

\subsection{Cohomology without invariant vectors}

In this subsection, we compile cohomological vanishing results for Zariski dense subgroups of a $\Q$-defined unipotent linear algebraic group $\U$ with respect to finite dimensional rational $\U$-modules without invariant vectors, in different cohomology theories. We start by mentioning a result by Dixmier, implying that such vanishings are already known for Lie algebra cohomology of $\mathfrak{u}_\Q$ where $\mathfrak{u}_\Q$ is the Lie algebra of $\U$. 

\begin{thm} \cite[Théorème 1]{dixmier1955cohomologie} \label{dixmier unipotent Lie algebra}
Let $\mathfrak{u}_\Q$ be a finite dimensional nilpotent Lie algebra over $\Q$. Let $M$ be some finite dimensional $\mathfrak{u}_\Q$-module such that $M^{\mathfrak{u}_\Q} = \{0\}$. Then $H^* (\mathfrak{u}_\Q, M) = \{0\}$.
\end{thm}

We now state a result relating rational cohomology with Lie algebra cohomology in the unipotent case. This is a particular case of a theorem that we will state later as Theorem \ref{T52}.

\begin{prop} \cite[Theorem 5.1]{hochschild} \label{unipotent Lie alg cohom vs rational cohom}
Let $\U$ be an irreducible  unipotent $\Q$-defined linear algebraic group with corresponding Lie algebra $\mathfrak{u}_\Q$. Then, for every finite dimensional rational $\U$-module $M$ and for all $n \geq 0$, we have the following isomorphisms
\begin{equation*}
    H^n(\mathfrak{u}_\Q, M) = H^n(\U, M).
\end{equation*}
In particular, if the $\U$-module $M$ has no $\U$-invariant vectors, then $ H^n(\U, M) = 0$.
\end{prop}

We now turn to prove analogous statements for group cohomology. In particular, we demonstrate that cohomology of a torsion-free $\Q$-linear nilpotent group $\Delta$ of finite abelian ranks with coefficients in a $\Q[\Delta]$-module without $\Delta$-invariant vectors is trivial. Notice that a subgroup $\Delta\ \leq \U(\Q)$ in some  unipotent $\Q$-defined linear algebraic group $\U$ is torsion-free and of finite abelian ranks.

\begin{prop}\label{prop: nilpotent wiouth invariant vectors cohomology}
    Let $\Delta$ be a torsion-free $\Q$-linear nilpotent group of finite abelian ranks, and let $M$ be a finite dimensional rational $\Delta$-module without $\Delta$-invariant vectors. Then $H^*(\Delta, M) = \{0\}$.
\end{prop}
\begin{proof}
We start by assuming that the group $\Delta$ is abelian. We proceed by induction on the dimension of the $\Q[\Delta]$-module $M$, and suppose that $M$ is a nontrivial one-dimensional representation. Thus, $M$ is given by a non trivial multiplicative character $\pi : \Delta \to \Q^*$. On homogeneous cochains we have an action given by
$$
\gamma c (\gamma_0, \ldots, \gamma_k) := c (\gamma_0 \gamma, \ldots, \gamma_k \gamma)
$$
for $\gamma, \gamma_0, \ldots, \gamma \in \Delta$. It is known that this action always induces a trivial action when passing to cohomology \cite[Corollary after Proposition 2]{hochschild-serre-group-extensions-cohomology}. This means that for every cocycle $c \in Z^i (\Delta, M)$, the element $\gamma.c - c$ is a coboundary for any $\gamma \in \Delta$. But since $\Delta$ is abelian, at the level of cocycles the same action is given by:
$$
\gamma c (\gamma_0, \ldots, \gamma_k) = c (\gamma \gamma_0, \ldots, \gamma \gamma_k) = \pi(\gamma)  c (\gamma_0, \ldots, \gamma_k).
$$
Since $\pi$ is a non-trivial character, there exists $\gamma \in \Delta$ such that $\pi(\gamma) \neq 1$, and hence, $ c = \frac{1}{(\pi (\gamma) - 1)} (\gamma.c - c)$ is a coboundary, implying $H^i (\Delta, M ) = 0$.

Now for the inductive case, suppose that $\Delta$ is a torsion-free, $\Q$-linear abelian group of finite abelian ranks and that $M$ is a $\Q[\Delta]$-module with no $\Delta$-invariant vectors and $\dim_{\mathbb{Q}}(M) > 1$. Then by tensoring with $\C$, we obtain a $\C[\Delta]$-module $M_\C : = M\otimes \C$ with no $\Delta$-invariant vectors such that $\dim_\C(M_\C) = \dim_\Q(M) > 1$. In particular, the $\C[\Delta]$-module $M_{\C}$ is not irreducible. Thus, there exists a nontrivial proper $\C[\Delta]$-submodule $N \leq M_\C.$ Since $M_\C^\Delta = \{0\},$ we have $N^{\Delta} = \{0\}.$ Additionally, we have a short exact sequence of $\C[\Delta]$-modules
$$
\{0\} \longrightarrow N \longrightarrow M_\C \longrightarrow M_\C/N \longrightarrow \{0\}
$$
which induces a long exact sequence in cohomology
\begin{eqnarray*}
\cdots \longrightarrow H^{i-1}(\Delta, M_{\C}/N) \longrightarrow H^i(\Delta, N) \longrightarrow H^i(\Delta, M_{\C})\\ \longrightarrow H^i(\Delta, M_{\C}/N) \longrightarrow H^{i+1}(\Delta, N) \longrightarrow \cdots.
\end{eqnarray*}
By induction, $H^i(\Delta,N) = 0$ for all $i$. Therefore, we may write the above long exact sequence as 
$$
\cdots \longrightarrow H^{i-1}(\Delta, M_{\C}/N) \longrightarrow 0 \longrightarrow H^i(\Delta, M_{\C}) \longrightarrow H^i(\Delta, M_{\C}/N) \longrightarrow 0 \longrightarrow \cdots.
$$
From here, we see that $H^i(\Delta,M_{\C}) \cong H^i(\Delta, M_{\C}/N).$

We claim that $M_{\C}/N$ doesn't have any $\Delta$-invariant vectors, and for a contradiction, suppose otherwise. There exists a $v \in M_{\C} \backslash N$ such that $a \cdot v = v \mod N$ for all $a \in \Delta$. Therefore, $a \cdot v - v \in N$ for all $a \in \Delta.$ Define the map $c \colon \Delta \to N$ to given by $c(a) = a \cdot v - v$ which can be seen to be a $1$-cocycle by a direct calculation. Since $N^{\Delta} = \{0\}$, we have $H^1(\Delta, N) =\{0\}.$ Therefore, $c$ is a 1-coboundary, that is, there exists a $w \in N$ such $\delta(w) = c$. Indeed, we see that $\delta(w)(a) = a \cdot w - w \in N.$ Hence, for all $a \in \Delta$, we have $a \cdot w - w = a \cdot v - v$. Rearranging, we have $a \cdot (w-v) = w-v$. By construction, $w-v \neq 0$ and $w-v \in M_{\C}^{\Delta}$ which is a contradiction. Hence, $(M_{\C}/N)^{\Delta} = 0$, and thus, induction implies $H^i(\Delta, M_{\C}/N) = \{0\}.$ Therefore, $H^i(\Delta,M_{\C}) = \{0\}$ for all $i$. Since $\C$ is a flat $\Q[\Delta]$-module, it follows $H^i(\Delta, M_{\C}) = H^i(\Delta,M) \otimes_{\Q} \C$ for all $i$. Hence, $ H^i(\Delta,M) = \{ 0 \}$ for all $i$. This concludes the proof when $\Delta$ is abelian.

Now suppose that $\Delta$ is a general torsion-free, $\Q$-linear nilpotent group of finite abelian ranks and that $M$ is an irreducible $\Q[\Delta]$-module.  We proceed by induction on step length, and observe that we justified the base case in the above arguments. Thus, suppose that $\Delta$ is not abelian and that $M^\Delta = \{0\}.$ Since $Z(\Delta)$ is a torsion-free abelian group, we have the following short exact sequence:
$$
1 \longrightarrow Z(\Delta) \longrightarrow \Delta \longrightarrow \Delta/Z(\Delta) \longrightarrow 1.
$$
We now consider the Lyndon-Hochschild-Serre spectral sequence associated to the above short exact sequence and write the terms of $E_2$ as
$$
E_2^{pq} = H^p(\Delta/Z(\Delta), H^q(Z(\Delta), M)) \Rightarrow H^{p+q}(\Delta, M).
$$
We claim that $M^{Z(\Delta)} = \{x \in M \: : \: z \cdot x =z \text{ for all } z \in Z(\Delta)\}$ is a $\Q[\Delta]$-submodule of $M$. Indeed, for $z \in Z(\Gamma)$, $m \in M^{Z(\Delta)}$ and $a \in \Delta,$ we have
$$
z \cdot (a \cdot m) = (za) \cdot m = a \cdot (z \cdot m) = a \cdot m.
$$
Thus, $a \cdot m$ is fixed by every $z \in Z(\Delta)$ which gives our claim. Since $M$ is irreducible, we have either $M^{Z(\Delta)} = \{0\}$ or $M^{Z(\Delta)} = M.$ If $M^{Z(\Delta)} = \{0\},$ then induction implies  $H^i(Z(\Delta), M)=\{0\}$ for all $i$. Hence, $E_2^{pq} = \{0\}$ for all $p,q.$ Therefore, $H^i(\Delta, M) =\{0\}$ for all $i$. Hence, we may assume that $M^{Z(\Delta)} = M.$ 

We now claim that
$$
H^q(Z(\Delta), M) = M^{d \choose q}
$$
where $Z(\Delta) \otimes_{\Z} \Q = \Q^d$. Since $M$ is a trivial $\Q[Z(\Delta)]$-module of finite dimension, the universal coefficient theorem allows us to write
$$
H^q(Z(\Delta), M) = H^q(Z(\Delta), \Q) \otimes_{\Q} M.
$$
Proposition \ref{prop:zariski dense cohomology trivial coefficients} implies
$$
H^q(\Q^d, \Q) \cong H^q(Z(\Delta), \Q)
$$
The well known formula for $H^k(\Q^d,\Q)$ is given as
$$
H^k(\Q^d,\Q) = \Q^{d \choose k}.
$$
Therefore,
$$
H^q(Z(\Delta), M) \cong \Q^{d \choose q} \otimes_{\Q} M = M^{d \choose q}
$$
as desired.

Since $M=  M^{Z(\Delta)}$ and $(M^{Z(\Delta)})^{\Delta / Z(\Delta)} = M^{\Delta} = \{ 0 \}$, we have that $\Delta/Z(\Delta)$ has an induced action on $M$ without invariant vectors. Therefore, induction implies
$$
H^p(\Delta/Z(\Delta), H^q(Z(\Delta), M)) = H^p(\Delta/Z(\Delta), M^{d \choose q}) = \bigoplus_{i=1}^{d \choose q}H^p(\Delta / Z(\Delta),M) = \{0\}.
$$
Hence, $H^p(\Delta,M) = \{0\}.$

If $M$ is a not an irreducible $\Q[\Delta]$-module, we proceed using the same long exact sequence arguments as in the abelian case. Therefore, we are finished.
\end{proof}

\subsection{Restriction map from irreducible unipotent linear algebraic groups defined over $\Q$ to Zariski dense subgroups}
In this subsection, we prove the following proposition that will be necessary for the proofs of Theorems \ref{main_thm} and \ref{resmap_solv} in the following section. 

\begin{prop}\label{unipotent_2}
    Let $\U$ be an irreducible unipotent $\Q$-defined linear algebraic group and $\Delta \leq \U(\Q)$ be a Zariski dense subgroup of finite abelian ranks. Then for every finite dimensional rational $\U$-module $M$, the restriction map $r_M \colon H^*(\U, M) \to H^*(\Delta, M)$ is an isomorphism of cohomology rings. 
\end{prop}

\begin{proof}
We use results from the previous subsection to extend Proposition \ref{prop:zariski dense cohomology trivial coefficients} to non-trivial coefficients. Since $\Delta$ is Zariski dense in $\U$, the $\Delta$-invariant vectors in $M$ coincide with $\U$-invariant vectors in $M$. We reason by induction on the dimension of the module $M$.

Suppose $\dim_\Q (M) = 1$. Then either the module $M$ is a trivial $\U$-module, or it has no invariant vectors. In the first case, the result follows from Proposition \ref{prop:zariski dense cohomology trivial coefficients}. In the second case, it follows from Propositions \ref{unipotent Lie alg cohom vs rational cohom} and \ref{prop: nilpotent wiouth invariant vectors cohomology}.

Now suppose that $\dim_\Q (M) \geq 2$ and that the result is known for any $\U$-module of dimension $< \dim_\Q(M) $. If the module $M$ has no invariant vectors, then again combining Propositions \ref{unipotent Lie alg cohom vs rational cohom} and \ref{prop: nilpotent wiouth invariant vectors cohomology} yields the result. Hence, we may suppose that $M^\U = M^\Delta \neq \{ 0 \}$. We have a short exact sequence of modules

    $$
    \{0\} \longrightarrow M^{\U} \longrightarrow M \longrightarrow M / M^{\U} \longrightarrow \{0\}
    $$
    where $M / M^{\U} = M /M^\Delta$ is given the induced $\U$-structure. The above sequence induces a long exact sequence in rational cohomology
    \begin{eqnarray*}
\cdots \longrightarrow H^{i-1}(\U, M/M^{\U}) \longrightarrow H^i(\U, M^{\U}) \longrightarrow\\ H^i(\U, M) \longrightarrow H^i(\U, M/M^{\U}) \longrightarrow H^{i+1}(\U, M^{\U}) \longrightarrow \cdots
\end{eqnarray*}
and in group cohomology
    \begin{eqnarray*}
\cdots \longrightarrow H^{i-1}(\Delta, M/M^{\Delta}) \longrightarrow H^i(\Delta, M^{\Delta}) \longrightarrow\\ H^i(\Delta, M) \longrightarrow H^i(\Delta, M/M^{\Delta}) \longrightarrow H^{i+1}(\Delta, M^{\Delta}) \longrightarrow \cdots.
\end{eqnarray*}
The restriction map from $\U$ to $\Delta$ induces a map between these two long exact sequences.

\adjustbox{scale=0.75,center}{
\begin{tikzcd}
	{ H^{i-1}(\U, M/M^{\U})} & {H^i(\U, M^{\U})} & {H^i(\U, M)}& {H^i(\U, M/M^{\U})}& {H^{i+1}(\U, M^{\U})} \\
	{H^{i-1}(\Delta, M/M^{\Delta})} & {H^i(\Delta, M^{\Delta})} & {H^i(\Delta, M)}& {H^i(\Delta, M/M^{\Delta})}& {H^{i+1}(\Delta, M^{\Delta})}
	\arrow["r", from=1-1, to=2-1]
    \arrow["r", from=1-2, to=2-2]
    \arrow["r", from=1-3, to=2-3]
    \arrow["r", from=1-4, to=2-4]
    \arrow["r", from=1-5, to=2-5] 
    \arrow["", from=1-1, to=1-2]
	\arrow["", from=1-2, to=1-3]
	\arrow["", from=1-3, to=1-4]
	\arrow["", from=1-4, to=1-5]
    \arrow["", from=2-1, to=2-2]
	\arrow["", from=2-2, to=2-3]
	\arrow["", from=2-3, to=2-4]
	\arrow["", from=2-4, to=2-5]
\end{tikzcd}}

By Proposition \ref{prop:zariski dense cohomology trivial coefficients} and the universal coefficient theorem for trivial $\U$-modules and $\Delta$-modules, the second and fifth vertical arrows are isomorphisms. Since $M^\U = M^\Delta \neq \{ 0 \}$, we have $\dim_\Q (M / M^\U) < \dim_\Q(M)$, and hence, the first and fourth vertical arrows are isomorphisms by induction hypothesis. Hence, the third vertical arrow is also an isomorphism by the 5-lemma. This completes the proof of the Proposition.
\end{proof}

\section{Proof of Main Theorems}\label{Section: Proof of main thms}

\subsection{Proof of Theorem \ref{resmap_solv}}

We dedicate this subsection to the proof of Theorem \ref{resmap_solv} which we restate here.

\begin{thm}
    Let $\G$ be an irreducible solvable $\Q$-defined linear algebraic group and $\Gamma \leq \G(\Q)$ be a Zariski dense subgroup. Then for every finite dimensional rational $\G$-module $M$, the restriction map $r_{M} \colon H^*(\G, M) \to H^*(\Gamma, M)$ is an injection of cohomology rings.
\end{thm}

Theorem \ref{resmap_solv} is an immediate consequence of Proposition \ref{unipotent_2}. 

\begin{proof}
Let $\G = \U \rtimes \T$ where $\U$ is the $\Q$-defined unipotent radical and $\T$ is a $\Q$-split maximal torus. \cite[Theorem 5.2]{hochschild} implies that $H^*(\G,M)$ is isomorphic to $H^*(\U,M)^\T$ where $H^*(\U,M)^\T$ is the subring of $\T$-invariants by the action of $\T$ on $\U$ by conjugation. Therefore, we have the natural inclusion $\alpha^*: H^*(\G,M) \rightarrow H^*(\U,M)$ which is induced by the inclusion $\alpha: \U \rightarrow \textbf{G}$. Letting $\Gamma_\U = \Gamma \cap \U(\Q)$, we also have by Lemma \ref{unipotent_2} that $r_M: H^*(\U, M) \to H^*(\Gamma_\U, M)$ is an isomorphism of cohomology rings. This gives the following  commutative diagram 

\begin{equation}\label{D8}
\begin{tikzcd}
{H^*(\G,M)=H^*(\U,M)^\T} \arrow[hookrightarrow]{rr}{\alpha^{\star}} \arrow[dd, "r_{M}"] && {H^*(\U,M)} \arrow[dd, "r_{M}", "\cong"'] \\ \\
{H^*(\Gamma, M)} \arrow[rr, "i^*"] && {H^*(\Gamma_\U, M)}
\end{tikzcd}
\end{equation}
where $i: \Gamma_\U \rightarrow \Gamma$ is the inclusion map. Since $r_M \circ \alpha^*$ is an injection, the restriction map from $\G$ to $\Gamma$ must be an injection. 
\end{proof}


\subsection{The ring map from $H^*(\mathfrak{g}_\Q,M)$ to $H^*(\Gamma, M)$}

We now turn to the proof of Theorem $\ref{main_thm}$ which we restate here.

\begin{thm}
       Let $\G$ be an irreducible solvable $\Q$-defined linear algebraic group  with associated $\Q$-defined Lie algebra $\mathfrak{g}_{\Q}$, and let $\Gamma \leq \G(\Q)$ be a Zariski dense subgroup such that it intersects the $\Q$-split maximal torus discretely in the Euclidean topology. Then for every finite dimensional rational $\G$-module $M$, there exists an isomorphism $\Phi_M: H^*(\mathfrak{g}_{\Q}, M) \to H^*(\Gamma, M)$ of cohomology rings.
\end{thm} 

We begin the proof by constructing a ring map $\Phi_M: H^*(\mathfrak{g}_\Q, M) \to H^*(\Gamma, M)$. We conclude with the proof that the ring map we construct is an isomorphism.

Let $\G = \U \rtimes \T$ be an irreducible solvable $\Q$-defined linear algebraic group where $\U$ is the $\Q$-defined unipotent radical and $\T$ is a $\Q$-split maximal torus, and let $\mathfrak{g}_\Q = \mathfrak{u}_\Q \rtimes \mathfrak{t}_\Q$ be the associated $\Q$-defined Lie algebra of $\G$. Let $\Gamma \leq \G(\Q)$ be a Zariski dense subgroup such that $\Gamma_\T = \Gamma \cap \T(\Q)$ is discrete in the Euclidean topology. Therefore, by definition, $\textbf{T} \cong ( \mathbb{G}_m^{\Q})^{\dim_{\Q}(\textbf{T})}$ and its Lie algebra $\mathfrak{t}_\Q$ is abelian. Since $\Gamma$ is Zariski dense in $\G(\Q)$, it follows that $\Gamma_\T$ is Zariski dense in $\T(\Q)$. Since $\Gamma_T$ is discrete in $\T(\Q)$, we have $\Gamma_{\T} \cong \Z^{\dim_{\Q}(\textbf{T})}$.  Additionally, we let $\Gamma_\U = \Gamma \cap \U(\Q)$, and subsequently, $\Gamma_\U$ is a Zariski dense subgroup of $\U$. 

The rational Mal'tsev completion of $\Z^{\dim_{\Q}(\T)}$ is the unique unipotent $\Q$-algebraic group $\N$ containing $\Z^{\dim_{\Q}(\T)}$ as a discrete cocompact subgroup. \cite[Theorem 2.1]{raghunathan} implies that $\Z^{\dim_{\Q}(\T)}$ is Zariski dense in $\N$. Let $\mathfrak{n}_{\Q}$ be the Lie algebra of $\N$ and $\mathfrak{n}_\Q^*$ be the differential graded algebra which is defined to be the graded $\Q$-algebra equipped with a differential map satisfying $d^2 =0$ and the Leibniz rule (see \cite{Weibel}, p. 112). It follows that $\mathfrak{t}_\Q$ is isomorphic to $\mathfrak{n}_\Q$. By \cite{lambe_priddy}, there exists a quasi-isomorphism $\bigwedge \mathfrak{n}_\Q^* \to A^*(B\Z^{\dim_{\Q}(\textbf{T})})$ where $\bigwedge \mathfrak{n}_\Q^*$ is the cochain complex considered as the space of right invariant differential forms on $\N/\Z^{\dim_{\Q}(\textbf{T})}$ and $A^*(B\Z^{\dim_{\Q}(\textbf{T})})$ is the $\Q$-polynomial de Rham complex of the classifying space $B\Z^{\dim_{\Q}(\textbf{T})}$ viewed as a differential graded algebra via the wedge product (see \cite{Kasuya}). Hence, there exists an induced isomorphism $H^*(\mathfrak{n}_{\Q}, \Q) \to H^*(\Z^{\dim_{\Q}(\textbf{T})}, \Q)$ of cohomology rings.  Upon choosing an isomorphism of $\Q$-vector spaces $\mathfrak{n}_{\Q} \to \mathfrak{t}_\Q$ and an isomorphism of groups $\Z^{\dim_{\Q}(\textbf{T})} \to \Gamma_{\T}$, it follows that there exists an induced isomorphism $\varphi^* \colon H^*(\mathfrak{t}_\Q, \Q) \to H^*(\Gamma_{\T}, \Q)$ of cohomology rings. Letting $\pi_\T: \G \to \T$ be the natural projection, composition with the pullback homomorphism $\pi_\T^*$ gives a ring map $\pi_\T^* \circ \varphi^*: H^*(\mathfrak{t}_\Q, \Q) \to H^*(\Gamma, \Q)$. 

Given a finite dimensional rational $\G$-module $M$, we recall that the restriction map $r_M \colon H^*(\G, M) \longrightarrow H^*(\Gamma, M)$ preserves the pairing $M \times \Q \overset{\cup}{\longrightarrow} M$ where $(v, \alpha) \to \alpha  v$ and hence is a ring map. It follows that there exists a ring map 

\begin{center}
$r_M\otimes(\pi_{\T}^* \circ \varphi^*): {H^*(\G, M) \otimes H^*(\mathfrak{t}_\Q, \Q}) \longrightarrow {H^*(\Gamma, M) \otimes H^*(\Gamma, \Q) }$. 
\end{center}

Given the cup product $H^*(\Gamma, M) \otimes H^*(\Gamma, \Q) \overset{\cup}{\longrightarrow} H^*(\Gamma, M)$, we then consider the composition of the above maps which we denote by $\Phi_M$:

\begin{equation}{\label{ringmap}}
\begin{tikzcd}
	{H^*(\G, M) \otimes H^*(\mathfrak{t}_\Q, \Q}) && {H^*(\Gamma, M) \otimes H^*(\Gamma, \Q) } && {H^*(\Gamma, M).}
	\arrow["{r_M\otimes(\pi_{\T}^* \circ \varphi^*)}", from=1-1, to=1-3]
	\arrow["{\Phi_M}"', curve={height=30pt}, from=1-1, to=1-5]
	\arrow["\cup", from=1-3, to=1-5]
\end{tikzcd}
\end{equation}

By \cite[Theorem 5.2]{hochschild}, there exists an isomorphism $\Lambda_M \colon H^*(\mathfrak{u}_\Q, M)^{\mathfrak{t}_\Q} \to H^*(\G, M)$ that is compatible with cup products such that the following diagram commutes:

\begin{equation}\label{D2}
\begin{tikzcd}
	{H^*(\mathfrak{u}_\Q, M)^{\mathfrak{t}_\Q}\otimes H^*(\mathfrak{u}_\Q,\Q)^{\mathfrak{t}_\Q}} && {H^*(\mathfrak{u}_\Q,M)^{\mathfrak{t}_\Q}} \\
	\\
	{H^*(\G,M) \otimes H^*(\G,\Q)} && {H^*(\G,M)}
	\arrow["\cup", from=1-1, to=1-3]
	\arrow["{\Lambda_M \otimes \Lambda_\Q}", "\cong"', from=1-1, to=3-1]
	\arrow["{\Lambda_M}", "\cong"', from=1-3, to=3-3]
	\arrow["\cup", from=3-1, to=3-3]
\end{tikzcd}.
\end{equation}

We note that $H^*(\mathfrak{u}_\Q, M)^{\mathfrak{g_{\Q}}} \cong H^*(\mathfrak{u}_\Q, M)^{\mathfrak{t}_\Q}$ since $H^*(\mathfrak{u}_\Q, M)^{\mathfrak{u}_\Q} \cong H^*(\mathfrak{u}_\Q, M)$. Therefore, by Theorem \ref{T31} and the irreducibility condition on $\G$, we have the following theorem.

\begin{thm}{\label{T52}}
Let $\G = \U \rtimes \T$ be an irreducible solvable $\Q$-defined linear algebraic group, and let $\mathfrak{g}_\Q = \mathfrak{u}_\Q \rtimes \mathfrak{t}_\Q$ be the associated $\Q$-defined Lie algebra of $\G$. Then for every finite dimensional rational $\G$-module $M$ and for all $n \geq 0$, the following isomorphism holds: 

\begin{center}
    $H^n(\mathfrak{g}_\Q, M) \cong \displaystyle\bigoplus_{i+j=n} H^i(\G, M) \otimes H^j(\mathfrak{t}_\Q, \Q)$.
\end{center}
\end{thm}

\begin{rem}
    By \cite[Théorème 2]{dixmier1955cohomologie}, we have $H^j(\mathfrak{t}_\Q, \Q) \neq 0$ for all $0 \leq j \leq \dim_\Q (\mathfrak{t}_\Q)$. In particular, this means that the cohomology space $H^n(\mathfrak{g}_\Q, M)$ always depends on all of the cohomology groups $H^i(\G, M)$ for $0 \leq i \leq n$.
\end{rem}

It then follows from Theorem \ref{T52} that there exists a cohomology map

\[ \begin{tikzcd}
H^*(\mathfrak{g}_\Q, M) \ar[r, "\cong"] & H^*(\G,M) \otimes H^*(\mathfrak{t}_\Q, \Q) \ar[r, "\Phi_M"] & H^*(\Gamma, M).
\end{tikzcd} \]

We now turn to the following proposition to show that this map is a map of cohomology rings. 

\begin{prop}
     Let $\G$ be an irreducible solvable $\Q$-defined linear algebraic group  with associated $\Q$-defined Lie algebra $\mathfrak{g}_{\Q}$, and let $\Gamma \leq \G(\Q)$ be a Zariski dense subgroup such that it intersects the $\Q$-split maximal torus discretely in the Euclidean topology. Then for every finite dimensional rational $\G$-module $M$, there exists a cohomology map 
    
    \begin{center}
    $\Phi_M: H^*(\mathfrak{g}_\Q, M) {\longrightarrow} H^*(\Gamma, M)$
    \end{center}
    that preserves cup products. In particular, the following diagram commutes:
   \begin{equation}\label{D3}
    \begin{tikzcd}
	{H^*(\mathfrak{g}_\Q,M) \otimes H^*(\mathfrak{g}_\Q,\Q)} && {H^*(\mathfrak{g}_\Q,M)} \\
	\\
	{H^*(\Gamma,M) \otimes H^*(\Gamma, \Q)} && {H^*(\Gamma,M)}
	\arrow["\cup", from=1-1, to=1-3]
	\arrow["{\Phi_M \otimes \Phi_{\Q}}", from=1-1, to=3-1]
	\arrow["{\Phi_M}", from=1-3, to=3-3]
	\arrow["\cup", from=3-1, to=3-3]
\end{tikzcd}
\end{equation}
where the pairing $M \times \Q {\longrightarrow} M$ is scalar multiplication and $\Phi_\Q: H^*(\mathfrak{g}_\Q, \Q) \xrightarrow{} H^*(\Gamma, \Q)$ is a ring map with trivial $\Q$-coefficients.
\end{prop}

\begin{proof}
    It suffices to check the commutativity of the above diagram. Let $x \otimes y \in H^i(\G,M) \otimes H^j(\mathfrak{t}_\Q, \Q)$ and $a \otimes b \in H^\ell(\G,\Q) \otimes H^k(\mathfrak{t}_\Q, \Q)$. We have the following:
    \begin{eqnarray*}
       \Phi_M(x \otimes y) \otimes \Phi_\Q(a \otimes b) &=& (r_M(x) \cup (\pi_\T^* \circ \phi^*)(y) )\otimes (r_{\Q}(a) \cup (\pi_\T^* \circ \phi^*)(b) )\\
       &=& (-1)^{j \ell} r_M(x) \cup r_{\Q}(a) \cup (\pi_\T^* \circ \phi^*)(y)  \cup (\pi_\T^* \circ \phi^*)(b)\\
       &=& (-1)^{j \ell} (r_M(x) \cup r_{\Q}(a)) \cup (\pi_\T^* \circ \phi^*)(y\cup b)\\
       &=& \Phi_M((-1)^{j \ell}(x \cup a) \otimes (y \cup b))\\
       &=& \Phi_M((x \otimes y) \cup (a \otimes b)). 
    \end{eqnarray*}
\end{proof}

\subsection{Proof of Theorem \ref{main_thm}}
Our goal is to show that the ring map 
$$\Phi_M: H^*(\mathfrak{g}_\Q, M) \to H^*(\Gamma, M)$$ 
that we constructed in the previous subsection is an isomorphism of cohomology rings. From Section \ref{Section: Spectral sequences}, we have corresponding spectral sequences $\{E_r^{pq}\}$ and $\{{E'_r}^{pq}\}$ converging to $H^*(\mathfrak{g}_\Q, M)$ and $H^*(\Gamma, M)$, respectively. If $\Phi_M$ is compatible with a map $f: E \to E'$ between spectral sequences, we refer to the following general theorem which gives conditions on $f$ for which $\Phi_M$ will be an isomorphism.  

\begin{thm}{(Comparison Theorem, \cite{Weibel}, p. 126)}{\label{compthm}} Let $\{E_r^{pq}\}$ and $\{{E'_r}^{pq}\}$ be spectral sequences that converge to $H^*$ and $H'^{*}$, respectively. Let $h: H^* \to H'^*$ be a map compatible with a morphism $f: E \to E'$ of spectral sequences. Then, if $f_r: E_r^{pq} \to {E'_r}^{pq}$ is an isomorphism for all $p,q$ and for some $r$ (hence for $r = \infty$ by the Mapping Lemma), then $h: H^* \to H'^*$ is an isomorphism.
\end{thm}

We can now prove Theorem \ref{main_thm}.

\begin{proof}
The short exact sequence of Lie algebras 
$$
1 \to \mathfrak{u}_\Q \to \mathfrak{g}_\Q \to \mathfrak{t}_\Q \to 1
$$
gives rise to the Hochschild-Serre spectral sequence: 
    
 \begin{center}
 $E_2^{pq} = H^p(\mathfrak{t}_\Q, H^q(\mathfrak{u}_\Q, M)) \Rightarrow H^{p+q}(\mathfrak{g}_\Q, M)$.
 \end{center}

Similarly, the group extension 
$$
1 \to \Gamma_\U \to \Gamma \to \Gamma_\T \to 1
$$
gives rise to the Lyndon-Hochschild-Serre spectral sequence: 

 \begin{center}
 ${E'_2}^{pq} = H^p(\Gamma_\T, H^q(\Gamma_\U, M)) \Rightarrow H^{p+q}(\Gamma, M)$.
 \end{center}

We will write $N =  H^q(\U, M)$, $n_\T = {\dim_{\Q}(\T)}$ and $n = \dim_\Q (N)$. Since the group $\T$ is semisimple, the $\T(\Q)$-action on $N$ can be diagonalized and expressed for $(x_1, \ldots, x_{n_\T}) \in  \T(\Q)$ as multiplication by the matrix
\begin{equation*}
    \begin{pmatrix}
        \beta_1^{\sum_{j = 1}^{n_\T} \gamma_{1j}x_j} & & \\
         & \ddots &  \\
         & &  \beta_n^{\sum_{j = 1}^{n_\T} \gamma_{n_\T j}x_j}\\
    \end{pmatrix}
\end{equation*}
where $\beta_{i} \in \Q_{>0}$ and $\gamma_{ij} \in \Q$ for every $(i, j) \in \{1 ,\ldots, n \} \times \{1, \ldots, n_\T \}$. With this expression we see that the subspace of invariant vectors $N^\T$ is spanned by the basis vectors $e_j$ for $j$ such that 
$$ \beta_n^{\sum_{j = 1}^{n_\T} \gamma_{n_\T j}x_j} =1.$$ 
Hence, we can split off $N^\T$ and obtain a $\T$-module $N'$ such that $N = N^\T \oplus N'$ and where $(N')^\T = \{0 \}$. 
By \cite[Theorem 5.1]{hochschild}, the exponential map $\exp: \mathfrak{u}_\Q \to \U$ induces an isomorphism of $\Q$-vector spaces $H^q(\mathfrak{u}_\Q, M) \cong H^q(\U, M)$.
The $\T$-action on the $\T$-module $H^q(\U, M)$ is induced by conjugation, so its differential gives a $\mathfrak{t}_\Q$-action on the same space. The $\mathfrak{t}_\Q$-action on $H^q(\mathfrak{u}_\Q, M)$ appearing in the Hochschild-Serre spectral sequence is induced by the adjoint action of $\mathfrak{t}_\Q$ on $\mathfrak{u}_\Q$.
Since $\exp \circ \mathrm{ad} = \mathrm{Ad} \circ \exp$, the exponential map induces an isomoprhism $H^q(\mathfrak{u}_\Q, M) \cong H^q(\U, M)$ of $\mathfrak{t}_\Q$-modules, and hence, it preserves a decomposition of the same type $$H^q(\mathfrak{u}_\Q, M) = H^q(\mathfrak{u}_\Q, M)^{\mathfrak{t}_\Q} \oplus N''. $$
Since $\T$ is irreducible, we have $(N'')^{\mathfrak{t}_\Q} = \{ 0 \}$ and $H^q(\mathfrak{u}_\Q, M)^{\mathfrak{t}_\Q} \cong H^q(\U, M)^{\T}$.
Hence, by Theorem \ref{dixmier unipotent Lie algebra} and additivity of cohomology in coefficients, we have
\begin{center}
 $E_2^{pq} \cong H^p(\mathfrak{t}_\Q, H^q(\mathfrak{u}_\Q, M)^{\mathfrak{t}_\Q})\oplus H^p(\mathfrak{t}_\Q,N'') = H^p(\mathfrak{t}_\Q, H^q(\mathfrak{u}_\Q, M)^{\mathfrak{t}_\Q})$.
\end{center}
Similarly, the $\T$-action on the $\T$-module $H^q(\U, M)$ restricts to a $\Gamma_\T$-action on the same space. Since the restriction map $r : H^q(\U, M)\cong H^q(\Gamma_\U, M)$ is an isomorphism by Proposition \ref{unipotent_2} (using that $\Gamma_\U$ is Zariski dense in $\U$), and since this map commutes with restriction of the $\T(\Q)$-action to a $\Gamma_\T$-action, it follows that $H^q(\U, M)\cong H^q(\Gamma_\U, M)$ as $\Gamma_\T$-modules. This means that $H^q(\Gamma_\U, M)$ preserves a decomposition of the same type as $H^q(\U, M)$, that is, $$H^q(\Gamma_\U, M) = H^q(\Gamma_\U, M)^{\Gamma_\T} \oplus N'''. $$
Since $\Gamma_\T$ is Zariski dense in $\T$, we have $H^q(\U, M)^{\T} \cong H^q(\Gamma_\U, M)^{\Gamma_\T}$ and also $(N''')^{\Gamma_\T} = \{ 0 \}$. Hence, using Proposition \ref{prop: nilpotent wiouth invariant vectors cohomology} we obtain
  \begin{center}
 ${E'_2}^{pq} \cong H^p(\Gamma_\T, H^q(\Gamma_\U, M)^{\Gamma_\T}) \oplus H^p(\Gamma_\T, N''') = H^p(\Gamma_\T, H^q(\Gamma_\U, M)^{\Gamma_\T})$.
 \end{center}

As we already noticed, combining, \cite[Theorem 5.1]{hochschild} and Proposition \ref{unipotent_2}, we have that $H^q(\mathfrak{u}_\Q, M) \cong H^q(\U, M) \cong H^q(\Gamma_\U, M)$ and that $$H^q(\mathfrak{u}_\Q, M)^{\mathfrak{t}_\Q} \cong H^q(\U, M)^{\T} \cong H^q(\Gamma_\U, M)^{\Gamma_\T}$$ as rational $\T$-modules (which is just a $\Q$-vector space isomorphism, since all actions are trivial).

On the other hand, \cite[Corollary 4.2]{lambe_priddy} says that $H^p(\mathfrak{t}_\Q, \Q) \cong H^p(\Gamma_\T, \Q)$, which by the universal coefficient theorem implies $H^p(\mathfrak{t}_\Q, M') \cong H^p(\Gamma_\T,M')$ for any trivial finite dimensional rational $\T$-module $M'$. By plugging $M' = H^q(\mathfrak{u}_\Q, M)^{\mathfrak{t}_\Q} \cong H^q(\Gamma_\U, M)^{\Gamma_\T}$ in this isomorphism, we obtain isomorphisms $f_2^{pq}: E_2^{pq} \to {E'_2}^{pq}$.

Following each of the isomorphisms shows that the maps $f_2^{pq}$ are in fact induced by maps 

\begin{center}
    $f_0^{pq}: E_0^{pq} =  C^p(\mathfrak{t}_\Q, C^q(\mathfrak{u}_\Q, M)) \to C^p(\Gamma_\T, C^q(\Gamma_\U, M)) = {E_0'}^{pq}$,
\end{center}
which in turn induce a map of spectral sequences $f_r^{pq}: E_r^{pq} \to {E_r'}^{pq}$ (which are all isomorphisms for $r \geq 2$). Recall that $E_r^{pq}$ and ${E'_r}^{pq}$ are first quadrant spectral sequences and thus, the filtrations on the cochain complexes are canonically bounded.

The maps defined on the total cochain complex as 
\begin{equation*}
    \bigoplus_{n= p + q} f_0^{pq}: \bigoplus_{n = p +q}  C^p(\mathfrak{t}_\Q, C^q(\mathfrak{u}_\Q, M)) \to \bigoplus_{n = p +q}  C^p(\Gamma_\T, C^q(\Gamma_\U, M)).
\end{equation*}
induce the ring map $\Phi_M: H^*(\mathfrak{g}_\Q, M) \to H^*(\Gamma, M)$  (\ref{ringmap}) when passing to cohomology.
Hence, by the classical convergence theorem (\ref{convthm}), the ring map $\Phi_M$ is compatible with the corresponding map $f: E \to E'$ of spectral sequences. This gives the following commutative diagram:

\begin{equation}{\label{D51}}
\begin{tikzcd}
E_2^{pq} \arrow[rr, "f_2^{pq}"] \arrow[dd, Rightarrow] && {E'_2}^{pq} \arrow[dd, Rightarrow]  \\ \\
H^*(\mathfrak{g}_\Q, M) \arrow[rr, "\Phi_M"] && H^*(\Gamma, M) 
\end{tikzcd}.
\end{equation}

We note that the cup product $H^*(\mathfrak{g}_\Q, M) \otimes H^*(\mathfrak{g}_\Q, \Q) \overset{\cup}\longrightarrow H^*(\mathfrak{g}_\Q,\Q)$ gives rise to an associated multiplicative structure $E_2^{pq} \times E_2^{p'q'} \to E_2^{p+p',q+q'}$ on $E_2^{pq}$ which is given by

\begin{eqnarray*}
H^p(\mathfrak{t}_\Q, H^q(\mathfrak{u}_\Q, M)) \otimes H^{p'}(\mathfrak{t}_\Q, H^{q'}(\mathfrak{u}_\Q, M)) &\overset{\cup}{\longrightarrow}&  H^{p+p'}(\mathfrak{t}_\Q, H^q(\mathfrak{u}_\Q, M) \otimes H^{q'}(\mathfrak{u}_\Q, M))\\ &\overset{\cup}{\longrightarrow}& H^{p+p'}(\mathfrak{t}_\Q, H^{q+q'}(\mathfrak{u}_\Q, M)).
\end{eqnarray*}

The cup product $H^*(\Gamma, M) \otimes H^*(\Gamma, \Q) \overset{\cup}\to H^*(\Gamma,\Q)$ gives rise to an analogous multiplicative structure on ${E'_2}^{pq}$. Therefore, the commutative diagram (\ref{D51}) is compatible with cup products.
 
By the comparison theorem (\ref{compthm}), it follows that $\Phi_M: H^*(\mathfrak{g}_\Q, M) \to H^*(\Gamma, M)$ is an isomorphism of cohomology rings.
\end{proof}

\bibliographystyle{plain}
\bibliography{ref}

@book {lennox_robinson,
    AUTHOR = {Lennox, John C. and Robinson, Derek J. S.},
     TITLE = {The theory of infinite soluble groups},
    SERIES = {Oxford Mathematical Monographs},
 PUBLISHER = {The Clarendon Press, Oxford University Press, Oxford},
      YEAR = {2004},
     PAGES = {xvi+342},
      ISBN = {0-19-850728-3},
   MRCLASS = {20F16 (20-02 20E15)},
  MRNUMBER = {2093872},
MRREVIEWER = {Patrizia Longobardi},
       DOI = {10.1093/acprof:oso/9780198507284.001.0001},
       URL = {https://doi.org/10.1093/acprof:oso/9780198507284.001.0001},
}

@article {raghunathan,
    AUTHOR = {Raghunathan, M. S.},
     TITLE = {Discrete subgroups of {L}ie groups},
   JOURNAL = {Math. Student},
  FJOURNAL = {The Mathematics Student},
      YEAR = {2007},
    NUMBER = {Special Centenary Volume},
     PAGES = {59--70 (2008)},
      ISSN = {0025-5742},
   MRCLASS = {22E40},
  MRNUMBER = {2527560},
MRREVIEWER = {O. V. Shvartsman},
}

@book {segal,
    AUTHOR = {Segal, Daniel},
     TITLE = {Polycyclic groups},
    SERIES = {Cambridge Tracts in Mathematics},
    VOLUME = {82},
 PUBLISHER = {Cambridge University Press, Cambridge},
      YEAR = {1983},
     PAGES = {xiv+289},
      ISBN = {0-521-24146-4},
   MRCLASS = {20-02 (20F16)},
  MRNUMBER = {713786},
MRREVIEWER = {John F. Bowers},
       DOI = {10.1017/CBO9780511565953},
       URL = {https://doi.org/10.1017/CBO9780511565953},
}

@book {borel,
    AUTHOR = {Borel, Armand},
     TITLE = {Linear algebraic groups},
    SERIES = {Graduate Texts in Mathematics},
    VOLUME = {126},
   EDITION = {Second},
 PUBLISHER = {Springer-Verlag, New York},
      YEAR = {1991},
     PAGES = {xii+288},
      ISBN = {0-387-97370-2},
   MRCLASS = {20-01 (20Gxx)},
  MRNUMBER = {1102012},
MRREVIEWER = {F. D. Veldkamp},
       DOI = {10.1007/978-1-4612-0941-6},
       URL = {https://doi.org/10.1007/978-1-4612-0941-6},
}

@book {springer,
    AUTHOR = {Springer, T. A.},
     TITLE = {Linear algebraic groups},
    SERIES = {Modern Birkh\"{a}user Classics},
   EDITION = {second},
 PUBLISHER = {Birkh\"{a}user Boston, Inc., Boston, MA},
      YEAR = {2009},
     PAGES = {xvi+334},
      ISBN = {978-0-8176-4839-8},
   MRCLASS = {20G15 (14L10)},
  MRNUMBER = {2458469},
}

@article {baues_grunewald,
    AUTHOR = {Baues, Oliver and Grunewald, Fritz},
     TITLE = {Automorphism groups of polycyclic-by-finite groups and
              arithmetic groups},
   JOURNAL = {Publ. Math. Inst. Hautes \'{E}tudes Sci.},
  FJOURNAL = {Publications Math\'{e}matiques. Institut de Hautes \'{E}tudes
              Scientifiques},
    NUMBER = {104},
      YEAR = {2006},
     PAGES = {213--268},
      ISSN = {0073-8301},
   MRCLASS = {20F28 (20G15 22E40)},
  MRNUMBER = {2264837},
MRREVIEWER = {Karel Dekimpe},
       DOI = {10.1007/s10240-006-0003-3},
       URL = {https://doi.org/10.1007/s10240-006-0003-3},
}

@article {Lie_cohomoloby,
    AUTHOR = {Hochschild, G. and Serre, J.-P.},
     TITLE = {Cohomology of {L}ie algebras},
   JOURNAL = {Ann. of Math. (2)},
  FJOURNAL = {Annals of Mathematics. Second Series},
    VOLUME = {57},
      YEAR = {1953},
     PAGES = {591--603},
      ISSN = {0003-486X},
   MRCLASS = {09.1X},
  MRNUMBER = {54581},
MRREVIEWER = {C.\ Chevalley},
       DOI = {10.2307/1969740},
       URL = {https://doi.org/10.2307/1969740},
}

@article{hochschild-serre-group-extensions-cohomology,
  title={Cohomology of group extensions},
  author={Hochschild, Gerhard and Serre, Jean-Pierre},
  journal={Transactions of the American Mathematical Society},
  volume={74},
  number={1},
  pages={110--134},
  year={1953},
  publisher={JSTOR}
}

@article {lambe_priddy,
    AUTHOR = {Lambe, Larry A. and Priddy, Stewart B.},
     TITLE = {Cohomology of nilmanifolds and torsion-free, nilpotent groups},
   JOURNAL = {Trans. Amer. Math. Soc.},
  FJOURNAL = {Transactions of the American Mathematical Society},
    VOLUME = {273},
      YEAR = {1982},
    NUMBER = {1},
     PAGES = {39--55},
      ISSN = {0002-9947,1088-6850},
   MRCLASS = {57T15 (17B56 22E25 58A12)},
  MRNUMBER = {664028},
       DOI = {10.2307/1999191},
       URL = {https://doi.org/10.2307/1999191},
}

@article {hochschild,
    AUTHOR = {Hochschild, G.},
     TITLE = {Cohomology of algebraic linear groups},
   JOURNAL = {Illinois J. Math.},
  FJOURNAL = {Illinois Journal of Mathematics},
    VOLUME = {5},
      YEAR = {1961},
     PAGES = {492--519},
      ISSN = {0019-2082},
   MRCLASS = {14.50 (17.30)},
  MRNUMBER = {130901},
MRREVIEWER = {W.\ T.\ van Est},
       URL = {http://projecteuclid.org/euclid.ijm/1255630894},
}

@article{dere_pengitore,
    author = {Dere, Jonas and Pengitore, Mark},
    title = {Automorphism groups of solvable groups of finite abelian ranks},
    journal = {arXiv preprint arXiv:2506.07991},
    year = {2025},
}

@article{kasuya1,
  title={Central theorems for cohomologies of certain solvable groups},
  author={Kasuya, Hisashi},
  journal={Transactions of the American Mathematical Society},
  volume={369},
  number={4},
  pages={2879--2896},
  year={2017}
}

@book {Weibel,
    AUTHOR = {Weibel, Charles A.},
     TITLE = {An introduction to homological algebra},
    SERIES = {Cambridge Studies in Advanced Mathematics},
    VOLUME = {38},
 PUBLISHER = {Cambridge University Press, Cambridge},
      YEAR = {1994},
     PAGES = {xiv+450},
      ISBN = {0-521-43500-5; 0-521-55987-1},
   MRCLASS = {18-01 (16-01 17-01 20-01 55Uxx)},
  MRNUMBER = {1269324},
MRREVIEWER = {Kenneth\ A.\ Brown},
       DOI = {10.1017/CBO9781139644136},
       URL = {https://doi.org/10.1017/CBO9781139644136},
}

@article {bause,
    AUTHOR = {Baues, Oliver},
     TITLE = {Infra-solvmanifolds and rigidity of subgroups in solvable
              linear algebraic groups},
   JOURNAL = {Topology},
  FJOURNAL = {Topology. An International Journal of Mathematics},
    VOLUME = {43},
      YEAR = {2004},
    NUMBER = {4},
     PAGES = {903--924},
      ISSN = {0040-9383},
   MRCLASS = {57S30 (22E25 57S25)},
  MRNUMBER = {2061212},
MRREVIEWER = {Karel\ Dekimpe},
       DOI = {10.1016/S0040-9383(03)00083-1},
       URL = {https://doi.org/10.1016/S0040-9383(03)00083-1},
}

@inproceedings {Kunkel,
    AUTHOR = {Kunkel, Paul J.},
     TITLE = {Rational cohomology of algebraic solvable groups},
 BOOKTITLE = {Proceedings of the {N}orthwestern conference on cohomology of
              groups ({E}vanston, {I}ll., 1985)},
   JOURNAL = {J. Pure Appl. Algebra},
  FJOURNAL = {Journal of Pure and Applied Algebra},
    VOLUME = {44},
      YEAR = {1987},
    NUMBER = {1-3},
     PAGES = {251--268},
      ISSN = {0022-4049,1873-1376},
   MRCLASS = {20G10 (22E40)},
  MRNUMBER = {885109},
MRREVIEWER = {K.\ Vogtmann},
       DOI = {10.1016/0022-4049(87)90029-6},
       URL = {https://doi.org/10.1016/0022-4049(87)90029-6},
}

@article {Kasuya,
    AUTHOR = {Kasuya, Hisashi},
     TITLE = {Extended simplicial rational {N}omizu's theorem and
              {S}ullivan's minimal models for non-nilpotent groups},
   JOURNAL = {Geom. Dedicata},
  FJOURNAL = {Geometriae Dedicata},
    VOLUME = {216},
      YEAR = {2022},
    NUMBER = {3},
     PAGES = {Paper No. 30, 9},
      ISSN = {0046-5755,1572-9168},
   MRCLASS = {55P62 (17B45 20F18 20G10 20G15 55R35 55U10)},
  MRNUMBER = {4404424},
MRREVIEWER = {Jean-Baptiste\ Gatsinzi},
       DOI = {10.1007/s10711-022-00691-w},
       URL = {https://doi.org/10.1007/s10711-022-00691-w},
}

@article {Mostow,
    AUTHOR = {Mostow, G. D.},
     TITLE = {Cohomology of topological groups and solvmanifolds},
   JOURNAL = {Ann. of Math. (2)},
  FJOURNAL = {Annals of Mathematics. Second Series},
    VOLUME = {73},
      YEAR = {1961},
     PAGES = {20--48},
      ISSN = {0003-486X},
   MRCLASS = {22.70},
  MRNUMBER = {125179},
MRREVIEWER = {Deane\ Montgomery},
       DOI = {10.2307/1970281},
       URL = {https://doi.org/10.2307/1970281},
}

@article {strongly_dense_free_subgroups,
    AUTHOR = {Kuranishi, Masatake},
     TITLE = {On everywhere dense imbedding of free groups in {L}ie groups},
   JOURNAL = {Nagoya Math. J.},
  FJOURNAL = {Nagoya Mathematical Journal},
    VOLUME = {2},
      YEAR = {1951},
     PAGES = {63--71},
      ISSN = {0027-7630,2152-6842},
   MRCLASS = {20.0X},
  MRNUMBER = {41145},
MRREVIEWER = {Deane\ Montgomery},
       URL = {http://projecteuclid.org/euclid.nmj/1118764740},
}

@article {commensurable_BS_groups,
    AUTHOR = {Casals-Ruiz, Montserrat and Kazachkov, Ilya and Zakharov,
              Alexander},
     TITLE = {Commensurability of {B}aumslag-{S}olitar groups},
   JOURNAL = {Indiana Univ. Math. J.},
  FJOURNAL = {Indiana University Mathematics Journal},
    VOLUME = {70},
      YEAR = {2021},
    NUMBER = {6},
     PAGES = {2527--2555},
      ISSN = {0022-2518,1943-5258},
   MRCLASS = {20E07 (20E34 20F05)},
  MRNUMBER = {4359918},
MRREVIEWER = {Matteo\ Vannacci},
       DOI = {10.1512/iumj.2021.70.9496},
       URL = {https://doi.org/10.1512/iumj.2021.70.9496},
}

@book {hatcher,
    AUTHOR = {Hatcher, Allen},
     TITLE = {Algebraic topology},
 PUBLISHER = {Cambridge University Press, Cambridge},
      YEAR = {2002},
     PAGES = {xii+544},
      ISBN = {0-521-79160-X; 0-521-79540-0},
   MRCLASS = {55-01 (55-00)},
  MRNUMBER = {1867354},
MRREVIEWER = {Donald\ W.\ Kahn},
}

@article{dixmier1955cohomologie,
  title={Cohomologie des algèbres de Lie nilpotentes},
  author={Dixmier, Jacques},
  journal={Acta sci. math. Szeged},
  volume={16},
  number={3--4},
  pages={246--250},
  year={1955}
}
\end{document}